\documentclass{amsart}
\usepackage{amssymb}
\usepackage{amsmath}
\usepackage{graphicx}
\usepackage{float}

\newtheorem{theorem}{Theorem}
\theoremstyle{plain}
\newtheorem{acknowledgement}{Acknowledgement}

\newtheorem{definition}{Definition}
\newtheorem{example}{Example}

\newtheorem{lemma}{Lemma}

\newtheorem{proposition}{Proposition}
\newtheorem{remark}{Remark}

\numberwithin{equation}{section}

\begin{document}
\title[Divergence and model adequacy]{Divergence and model adequacy, a
semiparametric case study}
\author{Michel Broniatowski}
\address{ Sorbonne Universit\'{e} and Universit\'{e} Paris Cit\'{e}, CNRS,
LPSM, F-75005 Paris, France}
\email{michel.broniatowski@sorbonne-universite.fr}
\author{Justin Moutsouka}
\curraddr{Sorbonne Universit\'{e}, F-75005 Paris, France}

\begin{abstract}
Adequacy for estimation between an inferential method and a model can be
defined through two \ main requirements: firstly the inferential tool should
define a well posed problem when applied to the model; secondly the
resulting statistical procedure should produce consistent estimators.
Conditions which entail these analytical and statistical issues are
considered in the context when divergence based inference is applied for
smooth semiparametric models under moment restrictions.\ A discussion is
also held on the choice of the divergence, extending the classical
parametric inference to the estimation of both parameters of interest and of
nuisance. Arguments in favor of the omnibus choice of the $L^{2}$ and
Kullback Leibler divergences as presented in \cite{Csiszar 1991} are
discussed and motivation for the class of power divergences defined in \cite%
{Basu et al.(1998)} is presented in the context of the present semi
parametric smooth models. A short simulation study illustrates the method.
\end{abstract}

\maketitle



\section{\protect\bigskip }

Introduction

Classical statistical inference deals with parametric models, which can be
seen as finite dimensional manifolds imbedded in the class of all
probability measures on some measurable space. Therefore to any value of the
parameter, a unique distribution in the model; based on a sample governed by
an unknown distribution in the model, many of associated inferential tools\
(for example maximum likelihood) have been studied extensively for a
considerable amount of models; obviously the choice of a parametric model
results from various sources (theoretical, convenience, rule of thumb,
habits, etc) and often cannot be stated as a truth pertaining to the
mechanism which may have generated the data (if such a mechanism exists);
therefore misspecification has to be considered and resulting properties of
the inference under misspecified models is a crucial step in the statistical
analysis. This question has some overlap with the issues pertaining to
robustness, which mainly focus on the role of so called outliers (or more
generally to artifacts in the sampling procedure). Misspecification issues
have led to consider tubes around models (in any topologically meaningful
sense for the statistical standpoint), with resulting properties of the
statistical procedures in this context; note that the inferential procedures
keep being fitted to the parametric setting, with no relevance to the
neighborhood of the model.

The situation gets even more complex when the model is a collection of
subsets of the class of all distributions with non void interior, and this
collection is indexed by a finite dimensional parameter, which is the
parameter of interest; a simple example is when each of these subsets
consists in all distributions with same expectation, which is the current
value of the parameter. This is a special case of the models considered in
this paper. Such models are named as "semiparametric models", since a
distribution in those is characterized through a finite parameter (of
interest), and an infinite dimensional parameter which captures all
characteristics of the distribution except the finite dimensional one.

How can statistical criterions handle the complexity of such a context,
taking into account the specificity of the infinite dimensional part of the
description of the model?\ Or phrasing differently, which is a reasonable
description of the model (in terms of regularity, or other) which still
makes inference on the finite dimensional parameter feasible through
standard parametric inferential tools, and how should the practical
inferential procedure be defined ?

We start with some outlook on the minimization of a pseudo distance between
the empirical measure defined by the data set and a model, defined loosely
as a collection of probability measures which we consider as candidates for
the generic distribution of the data set.\ This framework is generally
referred to as a "divergence based approach"; according to the choice of the
divergence (or "pseudo distance"), many classical methods for estimation and
testing can be recovered; see \cite{Pardo 2006}.

Before entering into our topics in a more detailed way, let us introduce
some preliminary definition on the context of this study.

As \ for the global notation, the space $\mathcal{X}$ \ which bears the data
is the euclidean space $\mathbb{R}^{m}$, endowed with its Borel field; all
involved probability measures are defined on $\left( \mathcal{X},\mathcal{B}%
\left( \mathcal{X}\right) \right) $ . In the sequel $\mathcal{M}^{1}$
designates the class of all probability measures defined on $\left( \mathcal{%
X},\mathcal{B}\left( \mathcal{X}\right) \right) $ and $\mathcal{M}%
^{1}(\lambda )$ the class of all elements in $\mathcal{M}^{1}$ which are
absolutely continuous (a.c) with respect to (w.r.t) the Lebesgue measure $%
\lambda $ on $\mathbb{R}^{m}.$

\subsection{Semi parametric models under moment conditions}

\noindent The models to be considered are defined in two ways.

\begin{itemize}
\item \ Firstly they are defined through constraints on moments; define $l$
linearly independent functions%
\begin{equation}
\left( \mathcal{X},\Theta \right) \ni \left( x,\theta \right) \rightarrow
g_{j}\left( x,\theta \right) \text{ \ }1\leq j\leq l  \label{Fonctions g}
\end{equation}%
where $\Theta $ is included in $\mathbb{R}^{d}$ and $l\leq d.$
\end{itemize}

For any $\theta $ let's denote by $\mathcal{M}_{\theta }$ the set of all
measures in $\mathcal{M}^{1}$ defined by 
\begin{equation}
\mathcal{M}_{\theta }:=\left\{ Q\in \mathcal{M}^{1}\text{ such that}\int
g_{j}(x,\theta )dQ(x)=0,1\leq j\leq l\right\}  \label{M teta}
\end{equation}%
Measures in $\mathcal{M}_{\theta }$ therefore satisfy $l$ linear
constraints. The model $\mathcal{M}$ is defined through 
\begin{equation}
\mathcal{M}=\cup _{\theta \in \Theta }\mathcal{M}_{\theta }  \label{Modele M}
\end{equation}

We further assume identifiability, meaning that $\mathcal{M}_{\theta }\cap 
\mathcal{M}_{\theta ^{\prime }}=\varnothing $ whenever $\theta \neq \theta
^{\prime }$.

\begin{itemize}
\item Secondly they are defined through some smoothness condition, which
substitutes the usual functional form of parametric inference.\ Therefore
all distributions in $\mathcal{M}$ share some common regularity condition,
which are characterized through regularity properties of their densities
with respect to the Lebesgue measure, to be stated in Section \ref{sect
smoothness}.
\end{itemize}

Models $\mathcal{M}$ satisfying the first set of the above conditions \ are
named as "\textit{moment constrained models}". When furthermore the second
set of condition is assumed we call $\mathcal{M}$ a "\textit{smooth moment
constrained model}".

\subsection{Divergences}

A divergence (or discrepancy) between two probability measures $P$ and $Q$
defined on the same measurable space $\mathcal{X}$ equipped with its Borel
field $\mathcal{B}\left( \mathcal{X}\right) $ is a non negative mapping 
\begin{equation*}
\left( P,Q\right) \rightarrow D\left( Q,P\right)
\end{equation*}%
such that $D\left( Q,P\right) =0$ if \ and only if $Q=P.$ No symmetry is
assumed, nor any triangular inequality; therefore a divergence need not be a
distance.\ Constructions of such functions $D$ are numerous; we briefly
sketch the present context leading to specific fields of applications in
statistics and learning. We refer to \cite{Br Stumm 2022} for description
and further references.

Let us introduce the following definition.

\begin{definition}
\label{Def adequacy}A divergence $D$ and a moment constrained model $%
\mathcal{M}$ satisfy adequacy when

(i) For any distribution $P_{0}$ such that $\inf_{Q\in \mathcal{M}}D\left(
Q,P_{0}\right) $ is finite, the problem 
\begin{equation*}
\arg \inf_{Q\in \mathcal{M}}D\left( Q,P_{0}\right)
\end{equation*}%
is a well posed problem

(ii) Given $P_{n}$ the empirical distribution of an i.i.d. sample under $%
P_{0}=P_{\theta _{T}}\in \mathcal{M}$ the estimator 
\begin{equation*}
\widehat{\theta }_{n}:=\arg \inf_{\theta \in \Theta }\inf_{Q\in \mathcal{M}%
_{\theta }}D\left( Q,P_{n}\right)
\end{equation*}%
is consistent in probability , and $\lim_{n\rightarrow \infty }\widehat{%
\theta }_{n}=\theta _{T}$.
\end{definition}

Therefore adequacy holds when conditions on $\mathcal{M}$ and on $D$ lead to
both above analytical and statistical properties.

\subsubsection{\protect\bigskip Decomposable divergences}

Consider a model $\mathcal{P\subset M}^{1}$ defined on $\mathcal{X}\mathbb{\ 
}.$ Say that a divergence 
\begin{equation*}
\left( \mathcal{P},\mathcal{M}^{1}\right) \ni \left( Q,P\right) \rightarrow
D\left( Q,P\right)
\end{equation*}%
is decomposable whenever there exist functionals $\mathfrak{D}^{0}:\mathcal{P%
}\mapsto \mathbb{R}$, $\mathfrak{D}^{1}:\mathcal{M}^{1}\mapsto \mathbb{R}$
and measurable mappings 
\begin{equation}
\rho _{Q}:\mathbb{R}^{d}\mapsto \mathbb{R},\text{ \ \ \ }  \label{24aa}
\end{equation}%
\ such that for all $Q\in \mathcal{P}$\ and some $P\in \mathcal{M}^{1}$ the
expectation $\int \rho _{Q}\mathrm{d}P$\ exists and 
\begin{equation}
D(Q,P)=\mathfrak{D}^{0}(Q)+\mathfrak{D}^{1}(P)+\int \rho _{Q}\mathrm{d}P.
\label{26a}
\end{equation}%
It is customary to restrict $P$ to the subset of $\mathcal{M}^{1}$ for which
the expectation $\int \rho _{Q}\mathrm{d}P$\ exists for all $Q$ in $\mathcal{%
P}$. Examples of decomposable divergences are numerous.\ Those include both
the L$^{2}$ and the Kullback Leibler divergences, but general Csiszar Ali
Silvey Morimoto divergences (CASM) (or so -called $f$-divergences) are not
captured through this definition; we refer to \cite{Broniatowski and Vajda}
and to \cite{Toma Broniatowski 2011} for definitions, examples and
properties; associated estimators are defined as minimizers of $D(Q,P_{n})$
upon $Q$, where $P_{n}$ designates the empirical distribution pertaining to
the observed data set $\left( X_{1},..,X_{n}\right) .$ Looking at (\ref{26a}%
) we see that decomposable divergences lead to simple M-estimators of $Q$
through substitution of $Q$ by $P_{n}$, with 
\begin{equation*}
Q_{n}:=\arg \min_{Q\in \mathcal{P}}\mathfrak{D}^{0}(Q)+\frac{1}{n}%
\sum_{i=1}^{n}\rho _{Q}(X_{i})
\end{equation*}%
whenever defined.

\begin{remark}
When $Q$ runs in a parametric family $\mathcal{P}:=\left\{ P_{\theta
},\theta \in \Theta \right\} $ then the function $\theta \rightarrow \rho
_{\theta }$ is reminiscent of the monotone embedding formalism for
generalized CASM or Bregman divergences; see \cite{Zhang Naudts} and
references therein, and corresponding formalism in generalized exponential
families under moment constraints in \cite{Br Kez studia} and \cite%
{Pelletier 2011} among others.
\end{remark}

\subsection{\protect\bigskip On the choice of the divergence}

\subsubsection{The need for a specific approach}

The identification of \ pertinent distances for inference in models defined
through moment conditions has been considered by Csiszar \cite{Csiszar 1991}%
; the general setting is that when all distributions involved share the same
finite support $K$ and when $\Theta $ is restricted to a single value; the
resulting ill posed inverse problem is somehow similar as the inferential
problem stated in the empirical likelihood paradigm with given moment
condition (hence for\textit{\ moment constrained models}); see next
paragraph for definitions, etc.\ \cite{Csiszar 1991} \ considers projection
rules defined through minimization of a pseudo distance between a given
distribution (the empirical distribution in the statistical context) and the
set of all probability vectors satisfying the moment constraint.\ Basic
assumptions which should be fulfilled by the admissible rules include the so
called "locality property": In relation with the present article, it states
that splitting $\mathcal{X}$ into two disjoint subsets $K_{1}$ and $K_{2}$
such that $\mathcal{X}=K_{1}\cup K_{2}$ ,the corresponding solutions of
those moment problems restricted to $K_{1}$ and $K_{2}$ with same constraint
can be assembled through mixing to produce the solution of the initial
moment problem on $\mathcal{X}$. Csiszar \cite{Csiszar 1991} identifies all
projection rules as pseudo distance minimization operators which satisfy the
locality property with some further natural axioms; those rules are
restricted to the $L^{2}$ projection operator, or to the Kullback-Leibler
operator (which yields the EL paradigm); although developed only for
finitely supported models, these arguments carry over to the continuous
case. In the semi parametric case considered here (namely the \textit{smooth
moment constrained model}), the locality property cannot be considered as a
necessary criterion for the definition of the projection rule; indeed the
global regularity constraint on the density of the solution to the moment
problem cannot generally be recovered through local ones: for example
assuming Lipschitz regularity of the densities of elements in $\mathcal{M}$
on $K_{1}$ and $K_{2}$ does not yield Lipschitz regularity on $\mathcal{X}.$
Henceforth, we are left with the choice of the projection rule, and the
quest for the incidence of its properties on the solution of the moment
problem under regularity assumptions remains open; this motivates this paper.

\bigskip

\subsubsection{\protect\bigskip Alternative procedures}

As mentioned earlier the question which we consider is the following:
starting with a discrepancy measure, which are the admissible models (in the
range of smooth moment constrained ones), for which optimization of the
given discrepancy is a valid procedure? For sake of completeness, we shortly
indicate some plausible alternative techniques, with indications about their
limitations.

The inference on $\theta $ in models defined by moment conditions can be
performed in a natural way for a number of statistical criterions. Indeed
for example for Cressie Read criterions, or more generally for CASM type
ones, a simple plug in of the empirical measure $P_{n}$ in place of $P$ in
the divergence $D(Q,P)$ allows to minimize it on $\mathcal{M}_{\theta }$ for
given $\theta $, and then to optimize upon $\theta .$ This is due to the
fact that the minimizer of $D(Q,P_{n})$ on $\mathcal{M}_{\theta }$ has
support included in the sample points $\left\{ X_{1},..,X_{n}\right\} $.
Therefore the seemingly formidable search for this minimization problem
boils down to a finite dimensional one, on the simplex of $\mathbb{R}^{n}.$
Such is the core argument for Empirical Likelihood (EL) methods and their
extensions (see e.g. \cite{Owen (2001)} or \cite{Br Keziou 2012 JSPI} for a
general CASM\ approach). All minimum empirical divergence methods (therefore
including EL) aim at assessing whether the model $\mathcal{M}$ is valid and
at the estimation of $\theta _{T}$ , the value of the parameter whenever $%
P_{0}$ which designates the distribution of the data equals $P_{\theta _{T}}$%
.$\ $ So they do not provide any knowledge on the density of $P_{\theta _{T}%
\text{ }}$ \ (whenever $P_{0}=P_{\theta _{T}}$ belongs to $\mathcal{M}$) nor
on the density of the projection of $P_{0}$ on $\mathcal{M}$ taking into
account the very definition of the model. Some penalized version of EL has
been proposed (see e.g. \cite{Tang Yang Zhao (2020)} and references therein)
but the context therein seems somehow different from ours. Extending the
parametric setting to a smoothed semiparametric one, it is possible to make
inference both on $\theta _{T}$ and on the density of $P_{\theta _{T}\text{ }%
}.$ We therefore take advantage of the very nature of the chosen criterion
to circumvent the obstacle due to the assumed regularity of the distribution
of the data.\ The same type of approach could be adopted making use of the
minimization of Bregman divergences or others .

Obviously for operational standpoint one could suggest to make use of method
as EL as\ a first step, hence making use of the Kullback Leibler projection
rule , leading to a distribution $Q_{\widehat{\theta }_{n}\text{ }}$%
supported by the sample $\left( X_{1},..,X_{n}\right) $ with $\widehat{%
\theta }_{n}$ converging to $\theta _{T}$ as $n$ tends to infinity, and then
projecting $Q_{\widehat{\theta }_{n}\text{ }}$ on the class of distributions
with density (w.r.t the Lebesgue measure) satisfying the prescribed
smoothness requirement. However this latest projection might lead to the
loss of the moment constraint; furthermore it bears a number of major
difficulties. These include the choice of a projection rule, which would
handle the absolute continuity obstacle, typically by making use of some
smoothing technique. There exists a huge literature on smoothed estimation
in non parametric or semi parametric models through penalization techniques,
out of the scope of this paper.

\subsection{A class of adequate divergences, the power divergences}

Because of the assumed regularity of the densities of measures in $\mathcal{M%
}$ we consider divergences $D(Q,P)$ which are explicit functionals of
densities, excluding therefore CASM divergences (except the Kullback-Leibler
one) for which the density $q:=dQ/d\lambda $ appears only through its ratio
with the density of $P.$ We turn therefore to the Bregman class , and
consider the subclass of power divergences $D_{\alpha }$ introduced by Basu
Hodjt, Harris and Jones (BHHJ) \cite{Basu et al.(1998)}, and which has been
embedded in a flexible family of divergences by Chicoki and Amari \cite%
{Cichocki Amari}; see also \cite{Basak basu extended bregman} for robust
Bregman divergences extending the BHHJ class, and \cite%
{EguciShintoKomori2022} for a comprehensive approach with applications. We
will stick to the basic form $D_{\alpha }$ which proves to be a pertinent
candidate for inference in parametric models. It also bears the benefit of
being indexed by a single parameter $\alpha $, which can be confronted with
the smoothness of the model. Also the power divergence $D_{\alpha }$ is
decomposable, which allows for simple application of classical results on
M-estimators obtained by plug in.

We briefly recall the main features of $D_{\alpha }$ which is defined through

\begin{equation*}
D_{\alpha }(Q,P):=\int \varphi (q(x),p(x))\mathrm{d}x
\end{equation*}%
where 
\begin{equation*}
\varphi (u,v)=\frac{1}{\alpha }u^{\alpha +1}-\left( 1+\frac{1}{\alpha }%
\right) u^{\alpha }\times v+v^{\alpha +1}.
\end{equation*}

Note that in accordance with some widely accepted notation in Information
theory, we denote $Q\in \Omega \rightarrow D_{\alpha }(Q,P)$ the projection
rule which maps the fixed measure $P$ in $\mathcal{M}^{1}(\lambda )$ over
some subset $\Omega $ of $\mathcal{M}^{1}(\lambda )$. This differs from the
original notation in \cite{Basu et al.(1998)}, where the notation is
reversed and the mapping is denoted $Q\rightarrow D_{\alpha }(P,Q).$ The
same notational ambiguity is unfortunately \ common in the global literature
on divergences, and leads to some confusion, for example between CASM
divergences and their conjugates (Neyman and Pearson Chi square, Kullback
Leibler and Likelihood divergences, etc).

\bigskip Indeed the BHHJ divergence is decomposable. In the semi or non
parametric context, it is more advisable to make use of a generic notation ,
namely

\begin{equation*}
\mathfrak{D}^{0}(Q):=\int q^{\alpha +1}d\lambda
\end{equation*}%
\begin{equation*}
\mathfrak{D}^{1}(P):=\frac{1}{\alpha }\int p^{\alpha +1}d\lambda
\end{equation*}%
\begin{equation*}
\rho _{q}:=-\left( 1+\frac{1}{\alpha }\right) q^{\alpha }.
\end{equation*}%
from which (\ref{26a}) holds.

Minimization on $Q$ over some class $\mathcal{M}_{\theta }$ included in $%
\mathcal{M}^{1}(\lambda )$ is equivalent to the minimization of the criterion%
\begin{equation}
R_{\alpha }(Q,P)=\mathfrak{D}^{0}(Q)+\int \rho _{q}dP  \label{R_alfa}
\end{equation}%
over $\mathcal{M}_{\theta },$ which allows for the plug in of $P_{n}$ $\ $\
in place of $P$, resulting in the common M-estimator framework.

We refer to \cite{Broniatowski and Vajda} and \cite{Toma Broniatowski 2011}
for definition, properties and extensions. 
We will consider values of $\alpha $ in $\left( 0,1\right] $ which ensures
that for \ all nonnegative $v$ the mapping $u\rightarrow \varphi (u,v)$
defined on $\mathbb{R}^{+}$ is strictly convex; the case when $\alpha =0$ is
the Kullback Leibler case, not considered here; the case when $\alpha =1$ is
the L$^{2}$ case, which is accessible through our approach.

The developed form of $D_{\alpha }(Q,P)$ is therefore 
\begin{equation}
D_{\alpha }(Q,P)=\int \left\{ q^{\alpha +1}(v)-\left( 1+\frac{1}{\alpha }%
\right) q^{\alpha }(v)p(v)+\frac{1}{\alpha }p^{\alpha +1}(v)\right\} \mathrm{%
d}v.~\text{ }  \label{D_alfa}
\end{equation}

The rationale for the BHHJ\ class in parametric inference in a model $%
\mathcal{P}:=\left\{ P_{\theta }\in \mathcal{M}^{1}(\lambda ),\theta \in
\Theta \right\} $ lies in the fact that whenever \ the integral in the above
display does not depend on the parameter $\theta ,$ as holds for location
models, then minimizing upon $\theta $ in $R_{\alpha }(P_{\theta },P_{n})$
amounts to smooth the usual likelihood score by a factor $p_{\theta
}^{\alpha -1}$ which damps the role of outliers in the estimating equation.

This procedure has been developed extensively and leads to classical limit
results for estimation and testing in parametric contexts; see Theorem 2 in 
\cite{Basu et al.(1998)}.\ The performance of this approach has been
compared to similar treatments making use of CASM divergences, both under
the model and under misspecification; globally speaking, performances of
either CASM divergence approach or power divergence approach are quite
similar (same limit distribution of the estimator and of the test statistics
as for the maximum likelihood approach (which falls in the field of CASM
divergences but not in the field of power ones for $\alpha $ in $\left( 0,1%
\right] $), nearly similar results in simulation runs on small or medium
size samples). Comparing properties between BHHJ divergences with various
values of $\alpha $ and corresponding ones for the power divergences of
Cressie-Read type with various parameters $\gamma $ (which describe the most
commonly used subclass in the CASM divergences) allows to obtain reasonably
robust estimators under contamination, as measures through the Influence
function; see \cite{Toma Broniatowski 2011}. These performances make them
good candidates for inferential tools in the semi parametric framework.

\subsubsection{\protect\bigskip Smooth semi parametric models under moment
conditions, specificity of the present approach}

\noindent Our standpoint is to propose a procedure which by its very nature
produces a smooth density which satisfies the model assumption. 
The drawback clearly lies in appropriate algorithms taking into account the
complexity of the required regularity of the model; a short simulation at
the end of the paper illustrates the behavior of the estimator in a very
simple case; however the present paper provides the necessary setup which
has to be developed, and which results as a common frame for similar
proposals under similar semi parametric framework; for example we may
consider classes of unimodal densities with unknown mode, or models with
densities defined by conditions on their L-moments\cite{Hosking}\cite{Br
decurninge}; in all those examples, the functional context is similar as the
one considered here; the class of densities imbedded in a function class
(denoted $E$ hereunder) has to be tailored accordingly. \bigskip

Consider the estimation of $\theta $ in $\mathcal{M}$; this yields to a two
steps minimization; the first one consists in the search for the minimizer $%
Q_{\theta }$ of $R_{\alpha }(Q,P_{n})$ for $Q$ in the smooth subset of $%
\mathcal{M}_{\theta }$, and the subsequent minimization should select the
value of $\theta $ which solves $\min_{\theta }$ $R_{\alpha }(Q_{\theta
},P_{n})$ where $Q_{\theta }$ solves the first minimization, whenever
possible.\ Firstly the model should be such that all minimization procedures
should be well defined; additional regularity assumptions on the model, with
respect to the variation of $\theta $ in $\Theta $, will be necessary in
order to perform the second optimization. The hypotheses in this paper are
not meant to meet the highest generality, but merely to address a simple
framework where adequacy can be considered and discussed.

The problem at hand writes therefore%
\begin{equation}
\widehat{\theta }_{n}:=\arg \min_{\theta \in \Theta }\min_{Q\in \mathcal{M}%
_{\theta }}R_{\alpha }(Q,P_{n}),  \label{def estim en 2 etapes}
\end{equation}%
where for all $\theta $, $\mathcal{M}_{\theta }$ consists in a family of
distributions with densities w.r.t the Lebesgue measure, with some
prescribed regularity. We need to introduce some description on the model;
this is done in the next Section.

\section{Notation and properties of the smooth semi parametric model}

\subsection{Constraints}

\noindent

All distributions in this model are defined on a compact subset $K$ of $%
\mathbb{R}^{m}.$The linearly independent functions $\left(
g_{1},...,g_{l}\right) $ introduced in (\ref{Fonctions g}) should satisfy
some basic requirements. Each of the functions $g_{l}$ is defined on $K$
with values in $\mathbb{R}.$ hence $g:=\left( g_{1},...,g_{l}\right) ^{T}$
is defined on $K\times \Theta $ with values in $\mathbb{R}^{l}.$ \ The
parameter space $\Theta $ is a compact subset in $\mathbb{R}^{d}$. \newline
We assume that for all $\theta $ the mapping 
\begin{equation}
\left( x,\theta \right) \longrightarrow g(x,\theta )\text{ \ \ is continuous
on \ \ }int(K)\times int\Theta .  \tag{G1}  \label{G1}
\end{equation}%
\newline
\newline
It follows that all functions $g_{l}$'s are uniformly bounded 
\begin{equation}
\sup_{\theta }\sup_{x\in K}\left\Vert g(x,\theta )\right\Vert <\infty 
\newline
\tag{G2}  \label{G2}
\end{equation}%
where $\left\Vert x\right\Vert $ designates the usual \ norm in $\mathbb{R}%
^{l}.$

It also follows from (\ref{G1}) that uniform continuity of $g$ holds in the
sense that as $\theta _{n}\rightarrow \underline{\theta }$ 
\begin{equation*}
\lim_{n\rightarrow \infty }\sup_{x\in K}\left\Vert g(x,\theta _{n})-g(x,%
\underline{\theta })\right\Vert =0.
\end{equation*}

%
%
%
%
%
%
%
%
%
%
%
%
%
%
%
%
%
%

\subsection{\label{sect smoothness}Regularity and smoothness assumptions%
\newline
}

The semi parametric model $\mathcal{M}_{\mathcal{E}}$ will be assumed to
consist in regular measures, in the sense that they should have density with
respect to the Lebesgue measure $\lambda $ on $K$, and that their densities
should be smooth. This is formalized as follows.

Let $\mathcal{M}_{K}^{1}$ be the class of all probability measures with
support $K$, and $\mathcal{M}_{K}^{1}\mathcal{(\lambda )}$ the class of all
probability measures in $\mathcal{M}_{K}^{1}$ which are a.c. w.r.t $\lambda
. $

We now define a subset $E$ of $\mathcal{C}_{b}\left( K\right) $ endowed with
the metric induced by the sup norm on $K$; for $q$ and $q^{\prime }$ in $E$,
denote%
\begin{equation*}
d(q,q^{\prime }):=\sup_{x\in K}\left\vert q(x)-q^{\prime }(x)\right\vert .
\end{equation*}
\newline
Two conditions will be assumed on $E$.\ 

1-We assume that $E$ is uniformly bounded on $K$, namely \newline

\begin{equation}
\sup_{q\in E}\sup_{x\in K}\left\vert q(x)\right\vert <+\infty .  \tag{E1}
\label{E1}
\end{equation}%
\newline

2- We denote by (E2) the following condition, which allows for standard
application of limit results for classes of functions in order to prove
consistency of the present estimation procedure.\ 
\begin{equation}
\sup_{q\in E}\sup_{x,y\in K}\frac{\left\vert q^{\alpha }(x)-q^{\alpha
}(y)\right\vert }{\left\vert x-y\right\vert }\leq M  \tag{E2}
\label{Lipschitz unif pour q alfa}
\end{equation}

(i)When the class $E$ is lower bounded on $K$ by some positive $\gamma $,
i.e. when 
\begin{equation*}
\inf_{q\in E}\inf_{x\in K}q(x)>\gamma >0
\end{equation*}

then we assume that $q^{\delta }$ is Lipschitz uniformly over $E$ for some $%
\delta \in \left( 0,1\right] $ ; this implies that $q^{\delta }$ is
Lipschitz for all $\delta $ in $\left( 0,1\right] $ and hence for $\alpha .\ 
$\ 

\bigskip (ii)If the class $E$ cannot be uniformly lower bounded on $K$ then
we assume that \textit{\ }$q^{\delta }\mathit{\ }$is Lipschitz uniformly on $%
E$\ for some $\delta :$\textit{\ }$0<\delta <\alpha .$

Under (E2)(i) (\ref{Lipschitz unif pour q alfa}) clearly holds, as under
(E2)(ii) since then \ $\sup_{p\in E}\sup_{x\in K}\left\vert \left( p^{\alpha
}\right) ^{\prime }(x)\right\vert $ is bounded.\newline
\newline

\begin{remark}
Condition (\ref{Lipschitz unif pour q alfa}) implies that the class $E$ is
equicontinuous: for all $\varepsilon >0$, there exists $\delta >0$ such that
for all $q$ in $E$,%
\begin{equation*}
\sup_{|x-x^{\prime }|<\delta }|q(x)-q(x^{\prime })|<\varepsilon .\newline
\end{equation*}
\end{remark}

\bigskip Define $\mathcal{E}$ the class of all non negative finite measures $%
Q$ on $K$ with density $q:=dQ/d\lambda $ in $E.$

For each $\theta $ consider the submodel

\begin{equation*}
\mathcal{M}_{\theta }:=\left\{ Q\in \mathcal{M}_{K}^{1}\text{ such that}\int
g(x,\theta )dQ(x)=0\right\} ,
\end{equation*}%
and its smooth counterpart 
\begin{equation*}
\mathcal{M}_{\theta _{\mathcal{E}}}:=\mathcal{M}_{\theta }\cap \mathcal{E},
\end{equation*}%
which we assume to be non void. We define the model $\mathcal{M}$ through 
\begin{equation*}
\mathcal{M}=\cup _{\theta \in \Theta }\mathcal{M}_{\theta }
\end{equation*}%
and the smooth version of $\mathcal{M}$ is defined by

\begin{equation*}
\mathcal{M}_{\mathcal{E}}=\cup _{\theta \in \Theta }\mathcal{M}_{\theta
}\cap \mathcal{E}=\cup _{\theta }\mathcal{M}_{\theta _{\mathcal{E}}}
\end{equation*}%
which we call the smooth moment constrained model.

As quoted before the first additional condition is an identifiability
property of the model $\mathcal{M}$ with respect to $\theta :$\newline
For $\theta \neq \theta ^{\prime }$, 
\begin{equation}
\mathcal{M}_{\theta }\cap \mathcal{M}_{\theta ^{\prime }}=\emptyset .\newline
\tag{M1}  \label{M1}
\end{equation}

We now state that for any $\theta $ all smooth densities in $\mathcal{M}%
_{\theta }$ can be distinguished from their counterparts in $\mathcal{M}%
_{\theta ^{\prime }}$ when $\theta ^{\prime }\neq \theta .$ This can be
phrased as follows: the collection of smooth submodels $\mathcal{M}_{\theta
_{\mathcal{E}}}$ is well separated \ in the sense that for any positive $%
\epsilon $ there exists some positive $\delta $ such that\newline
\begin{equation}
\left( d(\theta ,\theta ^{\prime })>\epsilon \right) \Rightarrow \left(
\inf_{\{Q\in \mathcal{M}_{\theta _{\mathcal{E}}},Q^{\prime }\in \mathcal{M}%
_{\theta _{\mathcal{E}}^{\prime }}\}}d(q,q^{\prime })>\delta \right) . 
\tag{M2}  \label{M2}
\end{equation}%
where we have denoted $q:=dQ/d\lambda $ and $q^{\prime }:=dQ^{\prime
}/d\lambda .$ The same notation will be used in the sequel: for example for $%
Q_{n}$ in $\mathcal{M}_{\mathcal{E}}$, $q_{n}$ will designate $%
dQ_{n}/d\lambda $, etc.

\begin{example}
$g(x)=x-\theta ,$ $\mathcal{M}_{\theta }=\{Q:\int_{K}x\mathrm{d}Q(x)=\theta
\}$ and clearly $\mathcal{M}_{\theta }\cap \mathcal{M}_{\theta }^{\prime
}=\emptyset .$ Whenever $\left\vert \int_{K}x(q(x)-q^{\prime
}(x))dx\right\vert >\varepsilon $ then $\int_{K}\left\vert q(x)-q^{\prime
}(x)\right\vert dx>\varepsilon /K$ , and therefore $d(q,q^{\prime })>\delta $
for some $\delta >0.$
\end{example}

\subsection{The estimator}

Given an i.i.d. sample $(X_{1},X_{2},...,X_{n})$ such that $X_{1}$ has
distribution $P_{\theta _{0}}$ $\in \mathcal{M}_{\theta _{0}\text{ }}$for
some $\theta _{0}\in \Theta $ we intend to provide an estimator for $\theta
_{0}$ minimizing the pseudo-distance between $P_{n}$ and $\mathcal{M}_{%
\mathcal{E}}$ where 
\begin{equation*}
P_{n}:=\frac{1}{n}\sum_{i=1}^{n}\delta _{X_{i}}
\end{equation*}%
is the empirical measure pertaining to the i.i.d. sample $%
(X_{1},X_{2},...,X_{n})$ . Note that the estimation is performed in \ the
smooth model $\mathcal{M}_{\mathcal{E}}$ and not in $\mathcal{M}.$

We introduce the estimator of $\theta _{0}$ in \ $\mathcal{M}_{\mathcal{E}}$
by 
\begin{equation}
\widehat{\theta }_{n}:=\arg \inf_{\theta }\inf_{Q\in \mathcal{M}_{\theta _{%
\mathcal{E}}}}D_{\alpha }(Q,P_{n}).  \label{EstimSmooth}
\end{equation}%
Formula (\ref{EstimSmooth}) provides an natural estimate of $\theta _{0}$ if 
$P_{\theta _{0}}\in \mathcal{M}_{\theta _{0_{\mathcal{E}}}}$.\ Indeed under
the identifiability conditions $(M1)$ and $(M2)$ we prove that the above
estimator converges to $\theta _{0}=\arg \inf_{\theta }\inf_{Q\in \mathcal{M}%
_{\theta }}D_{\alpha }(Q,P_{\theta _{0}})$;( see Theorem 1 and Theorem 9 ).%
\newline
In the alternative case that $P_{\theta _{0}}\in \mathcal{M}_{\theta _{0}}$
but $P_{\theta _{0}}\notin \mathcal{E}$ then formula (\ref{EstimSmooth})
defines an estimator of some $\tilde{\theta}:=\arg \inf_{\theta }\inf_{Q\in 
\mathcal{M}_{\theta _{\mathcal{E}}}}D_{\alpha }(Q,P_{\theta _{0}})$. Hence $%
P_{\tilde{\theta}}$ is the $D_{\alpha }-$projection of $P_{\theta _{0}}$ on $%
\mathcal{M}_{\mathcal{E}}$, and $\tilde{\theta}$ may be different from $%
\theta _{0}$ but still $P_{\tilde{\theta}}$ represents a proxy of $P_{\theta
_{0}}$ in the smooth model. We will consider a natural condition which
entails that $\tilde{\theta}=\theta _{0}$; (see Theorem 1).

\subsection{Adequacy}

For $D_{\alpha }$ and $\mathcal{M}_{\mathcal{E}}$ adequacy, since (\ref{26a}%
) holds and making use of $R_{\alpha }$ defined in (\ref{R_alfa}), Defnition %
\ref{Def adequacy} takes the following form

\begin{definition}
\label{Def adequacy R_alfa}The power divergence $D_{\alpha }$ and the smooth
moment constrained model $\mathcal{M}_{\mathcal{E}}$ satisfy adequacy when

(i) For any distribution $P_{0}$ such that $\inf_{Q\in \mathcal{M}_{\mathcal{%
E}}}D_{\alpha }\left( Q,P_{0}\right) $ is finite, the problem 
\begin{equation*}
\arg \inf_{Q\in \mathcal{M}_{\mathcal{E}}}D_{\alpha }\left( Q,P_{0}\right)
=\arg \inf_{Q\in \mathcal{M}_{\mathcal{E}}}R_{\alpha }\left( Q,P_{0}\right)
\end{equation*}%
is a well posed problem

(ii) Given $P_{n}$ the empirical distribution of an i.i.d. sample under $%
P_{0}=P_{\theta _{T}}\in \mathcal{M}_{\mathcal{E}}$ the estimator 
\begin{equation*}
\widehat{\theta }_{n}:=\arg \min_{\theta \in \Theta }\min_{Q\in \mathcal{M}%
_{\theta _{\mathcal{E}}}}R_{\alpha }(Q,P_{n})
\end{equation*}%
is consistent in probability , and $\lim_{n\rightarrow \infty }\widehat{%
\theta }_{n}=\theta _{T}$.
\end{definition}

\section{Projection and regularization}

We denote $P_{0}$ the distribution of the variable $X_{1}$. In this section
we consider both cases $P_{0}\in \mathcal{M}_{\theta _{0}}$ and $P_{0}\in 
\mathcal{M}_{\theta _{0_{\mathcal{E}}}}$for some $\theta _{0}$.\newline

Suppose that the following condition holds 
\begin{equation}
\inf_{Q\in \mathcal{M}_{\theta _{0_{\mathcal{E}}}}}D_{\alpha
}(Q,P_{0})<\inf_{Q\in \mathcal{M}_{\theta _{\mathcal{E}}}}D_{\alpha
}(Q,P_{0})  \tag{M3}  \label{cond regul par rapp a E}
\end{equation}%
for all $\theta \neq \theta _{0}$ , whenever $P_{0}$ belongs to $\mathcal{M}%
_{\theta _{0}}$ 
which formalizes the fact that $P_{0}$ is approximated smoothly with a
better score in $\mathcal{M}_{\theta _{0_{\mathcal{E}}}}$ than in any $%
\mathcal{M}_{\theta _{\mathcal{E}}}$, whenever $P_{0}$ belongs to $\mathcal{M%
}_{\theta _{0}}.$ Condition (\ref{cond regul par rapp a E}) connects the
smoothness condition of the model with the divergence criterion.\ It implies
that projecting $P_{0}$ on $\mathcal{M}$ or on $\mathcal{M}_{\mathcal{E}}$
identifies $\theta _{0}$ in a unique way, as stated in the following result,
to be proved in the Appendix.

\begin{theorem}
\label{theoreme 1} Under (\ref{cond regul par rapp a E}) it holds, whenever $%
P_{0}$ belongs to $\mathcal{M}$ or to $\mathcal{M}_{\mathcal{E}},$ 
\begin{equation}
\theta _{0}=\arg \inf_{\theta }\inf_{Q\in \mathcal{M}_{\theta _{\mathcal{E}%
}}}D_{\alpha }(Q,P_{0})=\arg \inf_{\theta }\inf_{Q\in \mathcal{M}_{\theta
}}D_{\alpha }(Q,P_{0}).  \label{Thm1}
\end{equation}
\end{theorem}

Before handling inference we need to explore some properties of minimum
pseudo-distance approximations in $\mathcal{M}_{\mathcal{E}}.$ We will make
use of a number of definitions, which we quote now. For fixed $P$ in $%
\mathcal{M}_{\mathcal{E}}$ the divergence $D_{\alpha }(.,P)|_{\mathcal{E}}$
is the restriction of $Q\rightarrow D_{\alpha }(Q,P)$ on $\mathcal{M}_{%
\mathcal{E}}$. \newline
For fixed $\theta $, let therefore the projection of $P$ on $\mathcal{M}%
_{\theta _{\mathcal{E}}}$ be 
\begin{equation*}
Q_{\theta }^{\ast }=\arg \inf_{Q\in \mathcal{M}_{\theta _{\mathcal{E}%
}}}D_{\alpha }(Q,P)|_{\mathcal{E}}
\end{equation*}%
whenever defined. \newline
Since for $Q\in \mathcal{M}_{\mathcal{E}}$ 
\begin{equation*}
D_{\alpha }(Q,P)|_{\mathcal{E}}=D_{\alpha }(Q,P)
\end{equation*}%
it holds%
\begin{equation*}
\arg \inf_{Q\in \mathcal{M}_{\theta _{\mathcal{E}}}}D_{\alpha }(Q,P)=\arg
\inf_{Q\in \mathcal{M}_{\theta _{\mathcal{E}}}}D_{\alpha }(Q,P)|_{\mathcal{E}%
}=Q_{\theta }^{\ast }.
\end{equation*}

We first set some general definition. 

\begin{definition}
\label{Definition 1}Let $\Omega $ be some subset of $\mathcal{M}^{1}$. The $%
\alpha -$divergence between the set $\Omega $ and a p.m. $P$ is defined by 
\begin{equation*}
D_{\alpha }(\Omega ,P):=\inf_{Q\in \Omega }D_{\alpha }(Q,P).
\end{equation*}%
A probability measure $Q^{\ast }\in \Omega $, such that $D_{\alpha }(Q^{\ast
},P)<\infty $ and 
\begin{equation*}
D_{\alpha }(Q^{\ast },P)\leq D_{\alpha }(Q,P)~\text{ for all }~Q\in \Omega
\end{equation*}%
is called a projection of $P$ on $\Omega $. This projection may not exist,
or may be not defined uniquely.
\end{definition}

\begin{definition}
\label{def2} The sequence of functions $q_{n}\in E$ tends to $q$ strongly if
and if 
\begin{equation*}
\sup_{x\in K}|q_{n}(x)-q(x)|\rightarrow0.
\end{equation*}
\end{definition}

Let $(Q_{n})_{n}\subset \mathcal{M}_{\mathcal{E}}$ ; if there exists some $q$
\ in $E$ such that 
\begin{equation}
\sup_{x\in K}|q_{n}(x)-q(x)|\rightarrow 0,  \label{q_n tend vers q}
\end{equation}%
then we say that $Q_{n}$ converges strongly to a non negative finite measure 
$Q$ such that $Q(A)=\int 1_{A}(x)q(x)\mathrm{d}x$ for all $A\in \mathcal{B}(%
\mathbb{R}^{m})$ .Denote $\left( Q_{n}{\underset{st}{\longrightarrow }}%
Q\right) $ when (\ref{q_n tend vers q}) holds.

\section{Projection: existence and uniqueness}

\label{Section3}

Let $P$ belong to $\mathcal{M}_{K}^{1}\left( \lambda \right) $ such that $%
\inf_{Q\in \mathcal{M}}$ $D_{\alpha }(Q,P)$ is finite.

We need some preliminary result pertaining to the properties of $\mathcal{M}%
_{\mathcal{E}}.$

\subsection{Closure of $\mathcal{M}_{\mathcal{E}}$\label{Subsect3}}

Conditions (E1) and (E2) imply that due to Arzela-Ascoli Theorem the set $E$
is pre-compact when endowed by the strong topology (see Definition \ref%
{Definition 1}). \newline
Let $(Q_{n})$ be a family of probability measures on $K$; it holds

\begin{proposition}
\label{theoreme 2 copy(1)} 
$\mathcal{M}_{\mathcal{E}}$ is relatively compact in $\mathcal{E}$ endowed
with the strong topology.
\end{proposition}

Let $\left\{ n_{j}\right\} $ $\subset \left\{ n\right\} $ and $\frac{%
dQ_{n_{j}}}{d\lambda }(x)=q_{n_{j}}(x)$,and $\sup_{x\in
K}|q_{n_{j}}(x)-q(x)|\longrightarrow 0$ then $(Q_{n_{j}})$ converges to some
p.m $Q$ and $Q(A)=\int_{A}q(x)\mathrm{d}\lambda (x)$ for all $A$ in $%
\mathcal{B}(K)$. 
\newline
\newline
Indeed 
\begin{eqnarray*}
\left\vert Q_{n_{j}}(A)-\int_{A}q(x)\mathrm{d}\lambda (x)\right\vert
&=&\left\vert \int 1_{A}(x)q_{n_{j}}(x)\mathrm{d}\lambda (x)-\int
1_{A}(x)q(x)\mathrm{d}\lambda (x)\right\vert \\
&\leq &\sup_{x\in K}\left\vert q_{n_{j}}(x)-q(x)\right\vert \lambda
(A)\longrightarrow 0.
\end{eqnarray*}%
So $(Q_{n_{j}})_{j\geq 1}$ converges to $Q$, such that $q(x)=\frac{dQ}{%
d\lambda }(x)$ .That $Q$ is a probability measure is a consequence of
Prohorov Theorem (see e.g. \cite{Billingsley1968}) since $(Q_{n})_{n\geq 1}$
is a tight family of p.m's . 

It follows that

\begin{theorem}
\label{theoreme 3} 
Under ($\ref{G1}$)$,$($\ref{G2}$) the set $\mathcal{M}_{\mathcal{E}}$ is
closed for the strong topology of convergence stated in Definition \ref%
{Definition 1}.
\end{theorem}

The proof of Theorem \ref{theoreme 3} is in the Appendix.

\subsection{Existence and uniqueness of the $D_{\protect\alpha }$-projection
of $P$ on $\mathcal{M}_{\mathcal{E}}$}

It holds

\begin{proposition}
\label{proposition 2} 
For any $\alpha \in \left( 0,1\right] $ the divergence function $Q\mapsto
D_{\alpha }(Q,P)$ from $\mathcal{M}_{K}^{1}(\lambda )$ to $[0,+\infty ]$ is
l.s.c \ for the strong topology.\newline
\end{proposition}

The proof of the above Proposition is in the Appendix.

Let $a>0$ and 
\begin{equation*}
A_{\mathcal{E}}(a):=\left\{ Q\in \mathcal{M}_{\mathcal{E}}:D_{\alpha
}(Q,P)\leq a\right\}
\end{equation*}%
be the $a-$level set of the divergence $Q\rightarrow D_{\alpha }(Q,P).$%
\newline

\begin{proposition}
\label{proposition 3} 
For all $a>0$, the level set $A_{\mathcal{E}}(a)$ of $Q\rightarrow D_{\alpha
}(Q,P)$ is compact in the strong topology.\ Furthermore for any $\theta $ in 
$\Theta $ 
\begin{equation*}
Q^{\ast }=\arg \inf_{Q\in \mathcal{M}_{\theta _{\mathcal{E}}}}D_{\alpha
}(Q,P).
\end{equation*}%
exists and is unique.
\end{proposition}

\begin{proof}
The set $F$ of functions $q$ in $E$ such that $D_{\alpha }(Q,P)\leq a$ for $%
Q $ with density $q$ is closed in $E$ by Proposition \ref{proposition 2};
now $cl(E)$ is compact by Arzela-Ascoli Theorem; hence $F$ is a compact
subset in $E$.

The mapping $q\rightarrow Q$ from $E$ to $\mathcal{M}_{\mathcal{E}}$ is
injective, whence $A_{\mathcal{E}}(a)$ is compact in $\mathcal{M}_{\mathcal{E%
}}.$ 
Let $a_{\theta }:=\inf_{Q\in \mathcal{M}_{\theta _{\mathcal{E}}}}D_{\alpha
}(Q,P)$ and let $\varepsilon >0.$ Then$\ $ $A_{\mathcal{E}}(a_{\theta
}+\varepsilon )\cap \mathcal{M}_{\theta _{\mathcal{E}}}\neq \emptyset .$

It can be observed that for all $\theta $ the set $\mathcal{M}_{\theta _{%
\mathcal{E}}}$ is a closed set, following the same arguments as in
Proposition \ref{proposition 2}. Since $\mathcal{M}_{\theta _{\mathcal{E}}}$
\ is closed and $A_{\mathcal{E}}(a_{\theta }+\varepsilon )$ is compact, $A_{%
\mathcal{E}}(a_{\theta }+\varepsilon )\cap \mathcal{M}_{\theta _{\mathcal{E}%
}}$ is compact.

Since 
\begin{equation*}
\arg \inf_{Q\in \mathcal{M}_{\theta _{\mathcal{E}}}}D_{\alpha }(Q,P)=\arg
\inf_{Q\in A_{\mathcal{E}}(a_{\theta }+\varepsilon )\cap \mathcal{M}_{\theta
_{\mathcal{E}}}}D_{\alpha }(Q,P),
\end{equation*}%
%
%
%
%
%
%
%
%
%
%
%
%
%
%
%
%
%
%
%
%
%
%
%
%
%
%
%
%
%
%
%
%
%
%
%
%
%
%
%
%
%
%
%
%
%
%
%
existence of the projection follows from the lower semi continuity of $%
Q\rightarrow D_{\alpha }(Q,P).$ Since $\varphi $ is strictly convex the
function $Q\in \mathcal{M}_{K}^{1}(\lambda )\rightarrow D_{\alpha }(Q,P)$ is
also strictly convex, and the projection of $P$ on any closed convex set $%
\Omega $ in $\mathcal{M}_{\theta _{\mathcal{E}}}$\ is uniquely defined.%
\end{proof}

Consider now the $D_{\alpha }$ -projection of $P$ on a convex subset $\Omega 
$ $\ $in $\mathcal{M}_{\mathcal{E}}.$ Making use of Propositions \ref%
{proposition 2} and \ref{proposition 3} it holds

\begin{theorem}
\label{theoreme 6} 
For any closed convex set $\Omega $ in $\mathcal{M}_{\mathcal{E}}$ the $%
D_{\alpha }$ projection of $P$ on $\Omega $ exists and is unique.
\end{theorem}

\begin{proof}
Indeed let 
\begin{equation*}
a:=\inf_{Q\in \Omega }D_{\alpha }(Q,P)
\end{equation*}%
and $\varepsilon >0.$ Then $A_{\mathcal{E}}(a+\varepsilon )\cap \Omega \neq
\emptyset $ . Since $\Omega $ is closed and $A_{\mathcal{E}}(a+\varepsilon ) 
$ is compact, existence of the projection follows.

\ Uniqueness is due to strict convexity . 
\end{proof}

\section{Minimum pseudo-distance estimator}

\label{Section4} Let $X_{1},...,X_{n}$ denote an i.i.d. sample of a random
vector $X\in \mathbb{R}^{m}$ with distribution $P_{0}$ in $\mathcal{M}%
_{K}^{1}(\lambda )$ . Let $P_{n}(.)$ be the empirical measure pertaining to
this sample, namely 
\begin{equation*}
P_{n}(.):=\frac{1}{n}\sum_{i=1}^{n}\delta _{X_{i}}(.),
\end{equation*}%
where $\delta _{x}(.)$ denotes the Dirac measure at point $x$. We define 
\begin{align*}
D_{\alpha }(\mathcal{M}_{\theta _{\mathcal{E}}},P_{0})& =\inf_{Q\in \mathcal{%
M}_{\theta _{\mathcal{E}}}}D_{\alpha }(Q,P_{0}) \\
& =\inf_{Q\in \mathcal{M}_{\theta _{\mathcal{E}}}}\left\{ \int \left(
q^{\alpha +1}(x)-\left( 1+\frac{1}{\alpha }\right) q^{\alpha }(x)p_{0}(x)+%
\frac{1}{\alpha }p_{0}^{\alpha +1}(x)\right) \mathrm{d}x\right\} .
\end{align*}%
Since optimization only pertains to $Q$ define in the following 
\begin{align*}
R_{\alpha }(\mathcal{M}_{\theta _{\mathcal{E}}},P_{0})& :=\inf_{Q\in 
\mathcal{M}_{\theta _{\mathcal{E}}}}R_{\alpha }(Q,P_{0}) \\
& =\inf_{Q\in \mathcal{M}_{\theta _{\mathcal{E}}}}\left\{ \int \left(
q^{\alpha +1}(x)-\left( 1+\frac{1}{\alpha }\right) q^{\alpha
}(x)p_{0}(x)\right) \mathrm{d}x\right\} ,
\end{align*}%
and the \textquotedblleft plug-in\textquotedblright\ estimate of $R_{\alpha
}(\mathcal{M}_{\theta _{\mathcal{E}}},P_{0})$ through%
\begin{align*}
\widehat{R}_{\alpha }(\mathcal{M}_{\theta _{\mathcal{E}}},P_{0})&
:=\inf_{Q\in \mathcal{M}_{\theta _{\mathcal{E}}}}R_{\alpha }(Q,P_{n}) \\
& =\inf_{Q\in \mathcal{M}_{\theta _{\mathcal{E}}}}\left\{ \int q^{\alpha
+1}(x)\mathrm{d}x-\left( 1+\frac{1}{\alpha }\right) \int q^{\alpha }(x)%
\mathrm{d}P_{n}(x)\right\} \\
& =\inf_{Q\in \mathcal{M}_{\theta _{\mathcal{E}}}}\left\{ \int q^{\alpha
+1}(x)\mathrm{d}x-\left( 1+\frac{1}{\alpha }\right) \frac{1}{n}%
\sum_{i=1}^{n}q^{\alpha }(X_{i})\right\}
\end{align*}%
In the same way,%
\begin{align*}
R_{\alpha }(\mathcal{M},P_{0})& :=\inf_{\theta \in \Theta }\inf_{Q\in 
\mathcal{M}_{\theta _{\mathcal{E}}}}R_{\alpha }(Q,P_{0}) \\
& =\inf_{\theta \in \Theta }\inf_{Q\in \mathcal{M}_{\theta _{\mathcal{E}%
}}}\left\{ \int q^{\alpha +1}(x)\mathrm{d}x-\left( 1+\frac{1}{\alpha }%
\right) \int q^{\alpha }(x)\mathrm{d}P_{0}(x)\right\}
\end{align*}%
can be estimated by 
\begin{equation*}
\widehat{R}_{\alpha }(\mathcal{M},P_{0}):=\inf_{\theta \in \Theta
}\inf_{Q\in \mathcal{M}_{\theta _{\mathcal{E}}}}\left\{ \int q^{\alpha +1}(x)%
\mathrm{d}x-\left( 1+\frac{1}{\alpha }\right) \frac{1}{n}\sum_{i=1}^{n}q^{%
\alpha }(X_{i})\right\}
\end{equation*}%
Since 
\begin{equation*}
\arg \inf_{Q\in \mathcal{M}_{\theta _{\mathcal{E}}}}D_{\alpha
}(Q,P_{0})=\arg \inf_{Q\in \mathcal{M}_{\theta _{\mathcal{E}}}}R_{\alpha
}(Q,P_{0})
\end{equation*}%
for any $\theta $ 
\begin{equation*}
\arg \inf_{Q\in \mathcal{M}_{\theta _{\mathcal{E}}}}R_{\alpha }(\mathcal{Q}%
,P_{0})
\end{equation*}%
exists and is unique (whether $P_{0}\in \cup \mathcal{M}_{\theta _{\mathcal{E%
}}}$ or not).\newline
We will consider estimators of $\theta _{0}$ where $P_{0}=P_{\theta _{0}}$
for some $\theta _{0}\in \Theta $ ; this corresponds to the fact that $%
P_{0}\in \mathcal{M}$. In this case by uniqueness of $\arg \inf_{\theta \in
\Theta }R_{\alpha }(\mathcal{M}_{\theta _{\mathcal{E}}},P_{0})$ and since
the infimum is reached at $\theta =\theta _{0}$ under the model, $\theta
_{0} $ is estimated through 
\begin{equation*}
\widehat{\theta }_{n}:=\arg \inf_{\theta \in \Theta }\inf_{Q\in \mathcal{M}%
_{\theta _{\mathcal{E}}}}\left\{ \int q^{\alpha +1}(x)\mathrm{d}x-\left( 1+%
\frac{1}{\alpha }\right) \frac{1}{n}\sum_{i=1}^{n}q^{\alpha }(X_{i})\right\}
.
\end{equation*}

\section{Asymptotic properties}

\subsection{Consistency}

\label{Section5} The pseudodistances $D_{\alpha }$ will be applied in the
standard statistical estimation model with i.i.d observations $%
X_{1},...,X_{n}$ governed by $P_{0}$ from a family $\mathcal{P}=\{P_{\theta
}:\theta \in \Theta \}$ $\subset \mathcal{M}_{K}^{1}(\lambda )$ of
probability measures on $(\mathbb{R}^{k},\mathcal{B}\left( \mathbb{R}%
^{k}\right) )$ indexed by a a set of parameters $\Theta \subset \mathbb{R}%
^{d}$ .

\begin{remark}
If $P_{0}\in $\textrm{$\mathcal{M}$} there exists an unique $P_{\theta
_{0}}\in \mathcal{M}$ such that $P_{0}=P_{\theta _{0}}\in \mathcal{M}$ ;
then by identifiability \textrm{%
\begin{equation*}
\arg \inf_{\theta }D_{\alpha }(P_{\theta },P_{\theta _{0}})=\theta _{0}.
\end{equation*}%
}
\end{remark}

In other words the unknown parameter $\theta _{0}$ is the unique minimizer
of the function $D_{\alpha }(P_{\theta },P_{0})$ 
\begin{equation}
\theta _{0}=\arg \min_{\theta }D_{\alpha }(P_{\theta },P_{\theta _{0}})\in
\Theta .  \label{equation1}
\end{equation}%
The empirical probability measures $P_{n}$ converge weakly a.s. to $P_{0}$
as $n\longrightarrow \infty $ . Therefore by plugging in (\ref{equation1})
the measures $P_{n}$ for $P_{0}$ one intuitively expects to obtain that the
estimator under the form 
\begin{equation*}
\arg \min_{\theta \in \Theta }M_{n}(P_{\theta },P_{n})
\end{equation*}%
converges to $\theta _{0}$ as $n\rightarrow \infty ,$ where $M_{n}(P_{\theta
},P_{n})$ is some empirical criterion which estimates the objective function 
$R_{\alpha }(P_{\theta },P_{0}).$\newline
We will repeatedly make use of a basic result which we recall for
convenience.\newline
Denote $M_{n}(\tau )$ a family of random functions of a parameter $\tau $
which belongs to a space $T$ endowed which a metric denoted $d$ .\newline
Assuming that the sequences $M_{n}$ converges uniformly to some
deterministic function $M$ defined on $T$, then the following result
provides a set of sufficient conditions which entail the weak convergence of
minimizers of $M_{n}$ to the minimizer of $M$ , if well defined.

\begin{lemma}
\label{lemme 3} 
(\cite{VanderVaart 1998}, theorem 5.7) Assume that

(1)$\sup_{\tau \in T}|M_{n}(\tau )-M(\tau )|{\overset{P}{\longrightarrow }}%
0, $

(2)For any $\epsilon >0,\inf_{\{t\in T,d(t,t_{0})\geq \epsilon
\}}M(t)>M(t_{0})$,

(3) the sequence $t_{n}$ satisfies 
\begin{equation*}
M_{n}(t_{n})\leq M_{n}(t_{0})+\circ _{p}(1)
\end{equation*}%
Then the sequence $t_{n}$ satisfies
\end{lemma}

\begin{equation*}
d(t_{n},t_{0}){\overset{P}{\longrightarrow }}0.
\end{equation*}%
Lemma \ref{lemme 3} will be used according to the context of minimization at
hand. \newline
By (\ref{def estim en 2 etapes}) we consider the inner and the outer
minimization problems leading to the estimator. This will be performed in
two steps: the inner minimization with respect to $Q$ in $\mathcal{M}%
_{\theta _{\mathcal{E}}}$ for fixed $\theta $, and the outer minimization
w.r.t $\theta .$\newline
Here we establish the consistency of the minimum pseudodistance estimator on
the closed set of measures a.c w.r.t $\lambda $ .

\subsubsection{Inner minimization: convergence of the projection of $P_{n}$%
\protect\bigskip on $\mathcal{M}_{\protect\theta _{\mathcal{E}}}$}

\noindent Fix $\theta \in \Theta .$ Denote%
\begin{equation*}
M_{n}(Q):=R_{\alpha }(Q,P_{n})
\end{equation*}%
where $Q\in \mathcal{M}_{\theta _{\mathcal{E}}}.$ \newline
\newline
Denote%
\begin{equation}
Q_{n}(\theta ):=\arg \inf_{Q\in \mathcal{M}_{\theta _{\mathcal{E}%
}}}R_{\alpha }(Q,P_{n}).  \label{q_n(teta)}
\end{equation}%
Existence and uniqueness of a p.m $Q_{n}(\theta )$ with density $%
q_{n}(\theta )$ follows from same arguments as in Proposition \ref%
{proposition 3}, substituting $P$ by $P_{n}$ . \newline
Denote accordingly the unique minimizer of $R_{\alpha }(Q,P_{0})$ on $%
\mathcal{M}_{\theta _{\mathcal{E}}}$, 
\begin{equation}
q_{\theta }^{\ast }:=\frac{dQ_{\theta }^{\ast }}{dP_{0}}\text{ where }%
Q_{\theta }^{\ast }:=\arg \inf_{Q\in \mathcal{M}_{\theta _{\mathcal{E}%
}}}R_{\alpha }(Q,P_{0}).  \label{q_teta*}
\end{equation}%
\newline
We prove that $q_{n}(\theta )$ converges to $q_{\theta }^{\ast }$ making use
of Lemma \ref{lemme 3}.\newline
Setting 
\begin{equation*}
M_{n}(\tau ):=R_{\alpha }(Q,P_{n}),
\end{equation*}%
with $\tau =\frac{dQ}{d\lambda }$, setting $d(\tau ,\tau ^{\prime
})=\sup_{x\in K}|q(x)-q^{\prime }(x)|$, it holds

\begin{lemma}
Fix $\theta .$\ Then Condition (1) in Lemma \ref{lemme 3} holds, namely%
\begin{equation*}
\sup_{Q\in \mathcal{M}_{\theta _{\mathcal{E}}}}\left\vert R_{\alpha
}(Q,P_{n})-R_{\alpha }(Q,P_{0})\right\vert \rightarrow 0\text{ in
probability.}
\end{equation*}
\end{lemma}

\begin{proof}
It holds%
\begin{equation*}
\sup_{Q\in \mathcal{M}_{\theta _{\mathcal{E}}}}|R_{\alpha
}(Q,P_{n})-R_{\alpha }(Q,P_{0})|\leq \left( 1+\frac{1}{\alpha }\right)
\sup_{Q\in \mathcal{M}_{\theta _{\mathcal{E}}}}\left\vert \frac{1}{n}%
\sum_{i=1}^{n}q^{\alpha }(X_{i})-E_{P_{0}}(q^{\alpha }(X))\right\vert
\end{equation*}%
which tends to $0$ almost surely as $n$ tends to infinity; indeed we may use
the arguments developed in \cite{Van der Vaart Wellner} pp 154-157, making
use of condition (\ref{Lipschitz unif pour q alfa}) , which implies, in its
notation, that the class of function $x\rightarrow q^{\alpha }$ belongs to $%
C_{M}^{\alpha }(K)$; therefore the class $E$ is Donsker, by Corollary 2.7.2.
in \cite{Van der Vaart Wellner}, henceforth is a Glivenko Cantelli class.
\end{proof}

We prove in the Appendix that the second condition in Lemma \ref{lemme 3}
holds

\begin{lemma}
\label{Lemma Cond 2 VdW} 
For any $\varepsilon >0$, 
\begin{equation*}
\inf_{\{Q:\Vert q-q_{\theta }^{\ast }\Vert >\epsilon ,Q\in \mathcal{M}%
_{\theta _{\mathcal{E}}}\}}R_{\alpha }(Q,P_{0})>R_{\alpha }(Q_{\theta
}^{\ast },P_{0}).\newline
\end{equation*}%
where $dQ/dP=q$ and d$Q_{\theta }^{\ast }/dP=q_{\theta }^{\ast }.$
\end{lemma}

We also state that the third condition in Lemma \ref{lemme 3} holds.

\begin{lemma}
\label{Lemma cond 3 VdW} 
\begin{equation*}
R_{\alpha }(Q_{n}(\theta ),P_{n})\leq R_{\alpha }(Q_{\theta }^{\ast
},P_{0})+o_{p}(1).
\end{equation*}
\end{lemma}

This follows from the very definition of $Q_{n}(\theta )$ for which $%
R_{\alpha }(Q_{n}(\theta ),P_{n})\leq R_{\alpha }(Q,P_{n})$ for all $Q\in 
\mathcal{M}_{\theta _{\mathcal{E}}}.$

Making use Lemma \ref{lemme 3} we have proved

\begin{theorem}
\label{theoreme 10} 
For any $\theta \in \Theta $,it holds, with $q_{n}(\theta )$ defined in (\ref%
{q_n(teta)}) and $q_{\theta }^{\ast }$ defined in (\ref{q_teta*}) 
\begin{equation*}
\sup_{x\in K}|q_{n}(\theta )(x)-q_{\theta }^{\ast }(x)|{\overset{P}{%
\longrightarrow }}0.
\end{equation*}
\end{theorem}

\subsubsection{Outer minimization}

We now consider the minimization in $\theta $ , with the following notation
. Let 
\begin{equation*}
\hat{\theta}_{n}:=\arg \inf_{\theta }\inf_{Q\in \mathcal{M}_{\theta _{%
\mathcal{E}}}}R_{\alpha }(Q,P_{n})=\arg \inf_{\theta }R_{\alpha
}(Q_{n}(\theta ),P_{n})
\end{equation*}%
and 
\begin{equation*}
{\theta }_{0}:=\arg \inf_{\theta }\inf_{Q\in \mathcal{M}_{\theta _{\mathcal{E%
}}}}R_{\alpha }(Q,P_{0})=\arg \inf_{\theta }R_{\alpha }(Q_{\theta }^{\ast
},P_{0}).
\end{equation*}%
The parameter \ $\theta _{0}$ such that $P_{0}=P_{\theta _{0}\text{ }}$is
defined in a unique way by the above display; indeed firstly note that ${%
\theta }_{0}$ is well defined, either when $P_{0}\in \mathcal{M}$ (i.e. $%
P_{0}=P_{\theta _{0}}$) (see Theorem \ref{theoreme 1}) or $P_{0}\notin 
\mathcal{M}$, in which case $P_{\theta _{0}}$ is the $D_{\alpha }-$%
projection of $P_{0}$ on $\mathcal{M}_{\mathcal{E}}.$ \newline
By Theorem \ref{theoreme 10}, we have proved that 
\begin{equation*}
\sup_{x\in K}|q_{n}(\theta )(x)-q_{\theta }^{\ast }(x)|{\overset{P}{%
\longrightarrow }}0.
\end{equation*}

where $q_{\theta }^{\ast }$ is defined in (\ref{q_teta*}). We want to show
that%
\begin{equation*}
\arg \inf_{\theta }R_{\alpha }(Q_{n}(\theta ),P_{n}){\overset{P}{%
\longrightarrow }}\arg \inf_{\theta }R_{\alpha }(Q_{\theta }^{\ast },P_{0}).
\end{equation*}%
where $Q_{\theta }^{\ast }=\arg \inf_{Q\in \mathcal{M}_{\theta _{\mathcal{E}%
}}}R_{\alpha }(Q,P_{0}).$

By definition 
\begin{equation*}
\hat{\theta}_{n}:=\arg \inf_{\theta }R_{\alpha }(Q_{n}(\theta ),P_{n})
\end{equation*}%
We prove that 
\begin{equation}
\arg \inf_{\theta }R_{\alpha }(Q_{\theta }^{\ast },P_{0})={\theta }_{0}.
\label{arginfR=teta0}
\end{equation}

Two cases may occur:

\begin{enumerate}
\item[(Case 1)] If $P_{0}\in \mathcal{M}$, i.e. if $P_{0}=P_{\theta _{0}}$
for some unique $\theta _{0}$ in $\Theta $, then (\ref{arginfR=teta0}) \
holds.

\item[(Case 2)] If $P_{0}\notin \mathcal{M}$,%
\begin{equation*}
{\theta }_{0}=\arg \inf_{\theta }\inf_{Q\in \mathcal{M}_{\theta _{\mathcal{E}%
}}}D_{\alpha }(Q,P_{0}).
\end{equation*}%
Therefore (\ref{arginfR=teta0}) holds.\ 
\end{enumerate}

We make use of Lemma \ref{lemme 3} with 
\begin{eqnarray}
M_{n}(\theta ) &:&=R_{\alpha }(Q_{n}(\theta ),P_{n}),
\label{Def M_n(teta), m(teta)} \\
M(\theta ) &:&=R_{\alpha }(Q_{\theta }^{\ast },P_{0}).\newline
\notag
\end{eqnarray}%
Indeed $\hat{\theta}_{n}$ converges to $\theta _{0}$ making use of Lemma \ref%
{lemme 3}.\newline
Set $q_{n}(\theta )(x):=\frac{dQ_{n}(\theta )}{d\lambda }(x)$ , and 
\begin{equation*}
d(q_{n}(\theta ),q_{\theta }^{\ast })=\sup_{x\in K}|q_{n}(\theta
)(x)-q_{\theta }^{\ast }(x)|;
\end{equation*}%
it then holds (see the proof in the Appendix)

\begin{proposition}
\label{proposition 4} 
Suppose that the following condition%
\begin{equation}
\sup_{\{Q\in \mathcal{M}_{\theta _{\mathcal{E}}},Q^{\prime }\in \mathcal{M}%
_{\theta _{\mathcal{E}}^{^{\prime }}},d(\theta ,\theta ^{\prime })<\delta
\}}d(q,q^{\prime })<C\delta  \tag{M4}  \label{K}
\end{equation}%
holds \ for some \ $C>0$ independent on $\theta $ and $\theta ^{\prime }$;
then 
\begin{equation*}
\sup_{\theta \in \Theta }\sup_{x\in K}|q_{n}(\theta )(x)-q_{\theta }^{\ast
}(x)|{\overset{P}{\longrightarrow }}0.
\end{equation*}%
%
%
%
%
%
%
%
%
%
%
%
%
%
%
%
%
%
%
%
%
%
%
%
%
%
%
%
%
%
%
%
%
%
%
%
%
%
%
%
%
%
%
%
%
%
%
%
%
%
%
%
\end{proposition}

\begin{lemma}
\bigskip \label{Lemma 8} Under Condition (\ref{K}) in Proposition \ref%
{proposition 4} , condition (1) in Lemma \ref{lemme 3} holds i.e. 
\begin{equation*}
\sup_{\theta \in \Theta }|M_{n}(\theta )-M(\theta )|{\overset{P}{%
\longrightarrow }}0
\end{equation*}%
with $M_{n}(\theta )$ and $M(\theta )$ defined in (\ref{Def M_n(teta),
m(teta)})
\end{lemma}

We now prove that the second condition in Lemma \ref{lemme 3} holds.

\begin{lemma}
\label{Lemma 9} For any $\varepsilon >0$, $\inf_{|\theta -\theta
_{0}|>\epsilon }M(\theta )>M(\theta _{0})$.
\end{lemma}

\begin{proof}
Denote $q_{\theta _{0}}^{\ast }$ \ the projection of $P_{0}$ on $\mathcal{M}%
_{\mathcal{E}}$, thus $\theta _{0}:=\arg \inf_{\theta \in \Theta }R_{\alpha
}(Q_{\theta }^{\ast },P_{0})$. For any $\theta \in \Theta ,$ let $Q_{\theta
}^{\ast }$ be the projection of $P_{0}$ on $\mathcal{M}_{\theta _{\mathcal{E}%
}}$; hence 
\begin{equation*}
R_{\alpha }(Q_{\theta }^{\ast },p_{0})\geq R_{\alpha }(Q_{\theta _{0}}^{\ast
},P_{0}).
\end{equation*}%
\newline
We prove that equality cannot hold in the above display. Let $|\theta
-\theta _{0}|>\epsilon $. \ Assume that there exists some $\theta _{1}$ with 
\begin{equation*}
d(q_{\theta _{1}}^{\ast },q_{\theta _{0}}^{\ast })>\delta
\end{equation*}%
$\ $such that 
\begin{equation}
R_{\alpha }(Q_{\theta _{1}}^{\ast },P_{0})=R_{\alpha }(Q_{\theta _{0}}^{\ast
},P_{0}),  \label{Thm11}
\end{equation}%
which cannot hold, since $\theta _{0}^{\ast }$ achieves the minimum of $%
R_{\alpha }(Q_{\theta }^{\ast },P_{0})$ on $\theta $, and $Q\longrightarrow
R_{\alpha }(Q,P_{0})$ is strictly convex.\ 
\end{proof}

We also prove the third condition in Lemma \ref{lemme 3}.

\begin{lemma}
\label{Lemma 10} $M_{n}(\theta )\leq M(\theta )+\circ_{p}(1)$.
\end{lemma}

\begin{proof}
By definition $M_{n}(\theta )<R_{\alpha }(Q_{\theta }^{\ast },P_{n})$ .%
\newline
Since $R_{\alpha }(Q_{\theta }^{\ast },P_{n})-R_{\alpha }(Q_{\theta }^{\ast
},P_{0}){\overset{P}{\longrightarrow }}0$ by Glivenko Cantelli Theorem, it
follows that 
\begin{equation*}
M_{n}(\theta )\leq R_{\alpha }(Q_{\theta }^{\ast },P_{0})+\eta _{n}
\end{equation*}%
for $n$ large enough,where $\eta _{n}{\overset{P}{\longrightarrow }}0.$ 
\end{proof}

As a consequence of the above arguments, the following convergence result
for the minimization of power type divergences on semiparametric models
defined by moment conditions holds.

\begin{theorem}
Under all the above conditions (E1), (E2), (M1), (M2), (M3) and (M4) it
holds, whenever $P_{0}$ $\ $belongs to $\mathcal{M}$ or $P_{0}$ belongs to $%
\mathcal{M}_{\mathcal{E}},$ with corresponding $\theta _{0},$%
\begin{equation*}
\lim_{n\rightarrow \infty }D_{\alpha }(\mathcal{M},P_{n})\rightarrow 0
\end{equation*}%
and 
\begin{equation*}
\lim_{n\rightarrow \infty }\widehat{\theta }_{n}=\theta _{0}
\end{equation*}%
Also we get 
\begin{equation*}
\lim_{n\rightarrow \infty }d\left( q_{\widehat{\theta _{n}}},p_{\theta
_{0}}\right) =0
\end{equation*}%
and all convergences above hold in probability.
\end{theorem}

\begin{remark}
Note that a sufficient condition for existence and uniqueness of the
projection of $P_{0}$ on $\mathcal{M}$ is obtainable under weaker condition
than (E2); however (E2) implies that $E$ is a Donsker class (hence a
Glivenko Cantelli class), which is a convenient argument for the convergence
of the estimator; in the same vein, this Donsker property should clearly
hold for the asymptotic distribution. Therefore (E2) seems a suitable nearly
unavoidable assumption.
\end{remark}

\subsection{A remark on the asymptotic distribution of the estimate}

By its very nature the $D_{\alpha }$ divergence is suited to statistical
inference in strictly parametric setting, as is the modified
Kullback-Leibler divergence, whose minimization amounts to maximum
likelihood. Recall that both coincide when $\alpha =0.$ In the case when $%
\alpha =0$, the likelihood estimating equation, assuming $P_{\theta }$ with
density $p_{\theta }$ writes 
\begin{equation}
\sum_{i=1}^{n}\acute{l}(X_{i}):=\sum_{i=1}^{n}\left( \frac{d}{d\theta }\log
p_{\theta }(X_{i})\right) _{\widehat{\theta }_{n}}=0  \label{ML equation}
\end{equation}%
where $\widehat{\theta }_{n}$ denotes the MLE, under suitable regularity
conditions.

Recall the general scheme leading to the asymptotic distribution of
estimators adapted to the present context: Assuming that the distribution $%
P_{\theta _{0}}$ of the data, denoted $P_{\theta _{0},q_{0}}$ (with $%
q_{0}:=dP_{\theta _{0}}/d\lambda $) belongs to the model \ $\mathcal{M}_{%
\mathcal{E}}$ and is embedded in a class of distributions $P_{\theta ,q}$
with $\theta \in \Theta $ and $q\in \mathcal{H}$, a Hilbert space of
functions defined on $K$ which contains $\mathcal{E}$. A classical method
amounts to substitute the classical score $\acute{l}=$ $\frac{d}{d\theta }%
\log p_{\theta }$ in the estimating equation (\ref{ML equation}) by the
efficient score $\widetilde{l}_{\theta ,q}$ , where $\widetilde{l}_{\theta
,q}(x):=\acute{l}_{\theta ,q}(x)-\Pi _{\theta ,q}($ $\acute{l}_{\theta ,q})$
where $\acute{l}_{\theta ,q}(x)$ denotes the parametric score function in
the semi parametric model $P_{\theta ,q}$ for $\theta $ when $q$ is fixed
and $\Pi _{\theta ,q}$ is the orthogonal projection onto the closure of the
nuisance score space for $q$.\ The Influence function of the resulting
efficient estimator $\widehat{\theta }_{n}$ is $\widetilde{l}_{\theta
_{0},q_{0}}$ which yields the asymptotic Gaussian approximation for $\sqrt{n}
$ $\left( \widehat{\theta }_{n}-\theta _{0}\right) $, with asymptotic
covariance matrix $\int $ $\widetilde{l}_{\theta _{0},q_{0}}^{T}$ $%
\widetilde{l}_{\theta _{0},q_{0}}dP_{\theta _{0},q_{0}}.$ We refer to \cite%
{VanderVaart 1998} for a precise account and examples, with explicit
techniques for the estimation of the asymptotic covariance of the
estimator.\ 

In our context, the estimator $\widehat{\theta }_{n}$ results from the two
steps optimization scheme, defined as inner optimization which for any $%
\theta $ provides $q_{n}(\theta )$ in $\mathcal{M}_{\mathcal{E}_{\theta }}$
and the outer optimization which yields $\widehat{\theta }_{n}.$ The case
when $\alpha \in 0,1]$ may be considered following a similar approach as in
the case $\alpha =0$ , but the description of the nuisance score space is
somehow more involved, since it amounts to consider the set of differentials
of the sub model $t\rightarrow P_{\theta _{t},q(\theta _{t})\text{ }}$ with $%
P_{\theta _{0},q(\theta _{0})}=P_{\theta _{0},q_{0}}$ \ along regular paths
at $t=0$, and to obtain the Influence function of $\widehat{\theta }_{n}$
through projection. This two steps procedure has been considered in the
econometric literature in the context of moment constrained optimization (of
regression type) with functional nuisance parameter; see \cite{Andrews 1994}%
, \cite{Ai Chen Econometrika 2003}, \cite{klaassen susyanto 2019} and
references therein. A convenient approach consists in approximating elements
in the nuisance space $\mathcal{H}$ by finite dimensional vectors (for
examples by sieves); see e.g. \cite{Tsiatis2006} for explicit treatment. \ 

A description of those asymptotics in the context of regression semi
parametric models is postponed to a future work.

\section{\protect\bigskip Estimating with polynomials}

A very simple toy case illustrates the present approach; consider a class of
polynomials $p(x)=ax^{2}$ $+bx+c\ $on $\left[ 0,1\right] $ which take
positive values on $\left[ 0,1\right] .$ Let $p_{0}(x)$ satisfy $%
\int_{0}^{1}p_{0}(x)dx=1$, $a=4$, and $\int_{0}^{1}xp_{0}(x)dx=\mu =.4.$ The
corresponding polynomial is positive on $\mathbb{R}$.\ Let $E$ be the class
of polynomials with coefficients close to those of $p_{0}$ and such that
both (E1) and (E2) hold (all elements in $E$ are bounded Lipschitz functions
on $\left[ 0,1\right] $). \ Regularity of elements in this class guarantees
this latest assertion.\ Furthermore $l=1$ and $g(x,\mu )=x-\mu .$ We
simulate $n$ points with density $p_{0}$ and choose $\alpha =1/2.$ The aim
of this exercise is to recover estimates of $a$ and $\mu .$

With $d=1$, let $g(x,\mu )=x-\mu $ with $\mu \in \left[ u,v\right] $, a
closed interval in $\mathbb{R}$, define the moment condition.

Conditions (M1) and (M2) are easily verified, with $\theta =\mu \in \left(
u,v\right) \ni .4$.\ 

The problem at hand here is therefore to find the value of $(a,\mu )$. For
the estimation of $\mu $ solve 
\begin{equation*}
\hat{\mu}=\arg \min_{\mu \in \Theta }\min_{Q\in \mathcal{M}_{\mu
_{E}}}R_{\alpha }(Q,P_{n})
\end{equation*}%
where $P_{n}$ will be obtained by sampling with the true parameters, for a
given sample size $n$. For any running value of $\mu $, say $\mu _{k}$ , the
minimization of $R_{\alpha }(Q,P_{n})$ with respect to all polynomials $q$
with degree less or equal 2, with integral 1, with positive values on $\left[
0,1\right] $, and satisfying $\int_{0}^{1}xq(x)dx=\mu _{k}$ provides $Q_{k}$
in $\mathcal{M}_{\mu _{_{k}E}}$, hence the inner optimization. \ Evaluation
of $R_{\alpha }(Q_{k},P_{n})$ on a grid of values $\mu _{k}$ provides the
outer optimization. Figures 1 and 2 hereunder capture the results; in Figure
2, we quote the estimate of $a$, since both constraints $%
\int_{0}^{1}q(x)dx=1 $ and $\int_{0}^{1}xq(x)dx=\mu $ provide $q$ for given $%
a.$

\begin{figure}[H] 
\includegraphics[width=350px]{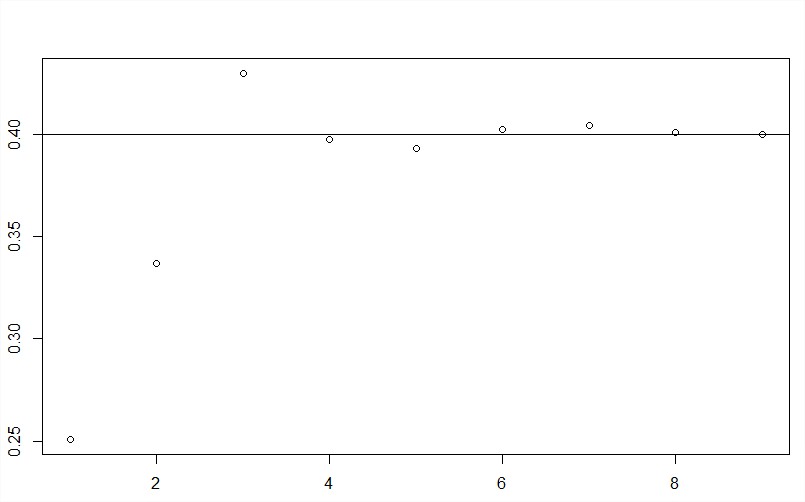}
\caption{Estimation of the parameter $\protect\mu$ for $%
n=10,50,100,500,1000,5000,10000,50000,100000$.The abscissa quotes 2 for n = 50, 4 for n = 500, 6 for n = 5000, 8 for n = 50000 \\[0.5cm]
}
\includegraphics[width=350px]{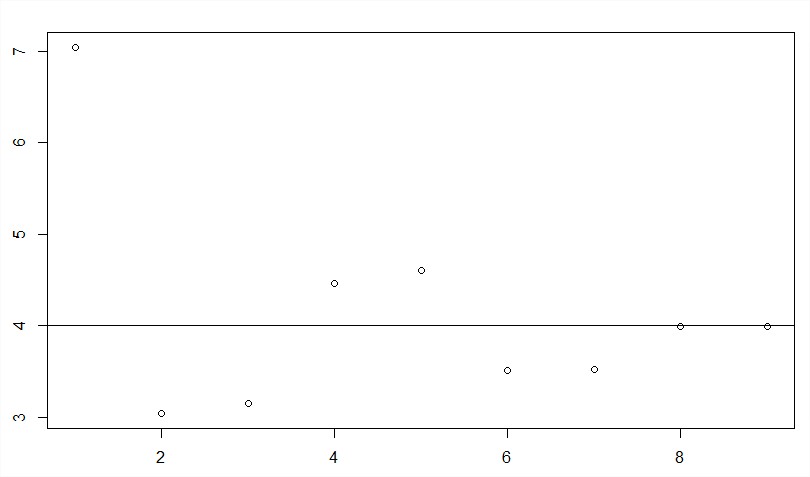}
\caption{Estimation of the coefficient $a$ for $%
n=10,50,100,500,1000,5000,10000,50000,100000$.The abscissa quotes 2 for n = 50, 4 for n = 500, 6 for n = 5000, 8 for n = 50000 \\[0.5cm]
}  
\end{figure}

\section{\protect\bigskip Conclusion}

Adding smoothness to moment constrained models introduces the need for
adequate inferential techniques; indeed the context underlying the fact that
the omnibus $L^{2}$ and Kullback Leibler\ divergences are the only valid
ones for models defined through moment constraints, as discussed in \cite%
{Csiszar 1991}, fall short under additional regularity requirements. This
introduces the need for divergence based approaches under alternative pseudo
distances.

In this paper we have stated a set of regularity conditions pertaining to a
smooth moment constrained model indexed by a finite dimensional parameter of
interest $\theta $ and a functional nuisance parameter $q$ in some function
class $E$ (which is the space of smooth densities of the model); those allow
for adequacy with some power divergences $D_{\alpha }$ with $0<\alpha \leq 1$%
; those conditions firstly validate the choice of $D_{\alpha }$ for the
inference through their analytical properties (implying existence and
uniqueness of $D_{\alpha }$ -projections on the model); furthermore they
entail consistency of the estimators . Condition (E2) appears as a good
compromise between analytic and statistical requirements, which we call
adequacy. The two steps optimization procedure produces consistent
estimators of both true parameters $\theta _{T}$ and $q_{T}$.

Adequacy holds as follows:

Either the class $E$ is lower bounded,and $q^{\delta }$ is Lipschitz
uniformly over $E$ for some $\delta \in \left( 0,1\right] $

or

the class $E$ cannot be uniformly lower bounded on $K$ but \textit{\ }$%
q^{\delta }$\textit{\ }is \ Lipschitz uniformly over\textit{\ }$E$\textit{\ }%
for some\textit{\ }$\delta :$\textit{\ }$0<\delta <\alpha $.

Additionally adequacy requires sharp requirements which enforce
identifiability, mainly strong separation between the submodels indexed by $%
\theta $, see (M1), (M2), (M3) and (M4); condition (M3) establishes a
connection between the structure of the model and the divergence.

Other semi parametric models of the same type frequently occur in the
econometric or reliability literature, for example when the nuisance
parameter consists in subsets of regular convex or monotone bounded
densities (see e.g.\cite{Van der Vaart Wellner} Chapter 3) , or models with
restricted hazard rate functions.

The limit distributions of the couple of parameters $\theta _{T}$ and $q_{T}$
are not handled in this paper and can be studied making use of nowadays
classical semiparametric inferential methods see \cite{VanderVaart 1998}\cite%
{Tsiatis2006}; however due to the two steps framework of our optimization
proposal, it could be wise to consider approximating schemes (sieves or
RKHS) for the nuisance parameter space. This is postponed to future work.

\bigskip

\ 

\section{Appendix}

\subsection{\protect\bigskip Proof of Theorem \protect\ref{theoreme 1}}

First case: Suppose that $P_{0}=P_{\theta _{0}}\in \mathcal{M}_{\mathcal{E}}$%
., i.e. $P_{0}$ $\in \mathcal{M}_{\theta _{0_{E}}}.$ Then%
\begin{equation*}
\inf_{Q\in \mathcal{M}_{\theta _{0_{\mathcal{E}}}}}D_{\alpha }(Q,P_{0})=0.
\end{equation*}%
.\newline
Since $\mathcal{M}_{\mathcal{E}}\supset \mathcal{M}_{\theta _{0_{\mathcal{E}%
}}}$, we have%
\begin{equation*}
\inf_{Q\in \mathcal{M}_{\mathcal{E}}}D_{\alpha }(Q,P_{0})=0.
\end{equation*}%
\newline
Furthermore, $\theta _{0}$ realizes $\inf_{Q\in \mathcal{M}_{\theta
_{0_{E}}}}D_{\alpha }(Q,P_{0})=0$.\newline
So%
\begin{equation*}
\theta _{0}\in \arg \inf_{\theta }\inf_{Q\in \mathcal{M}_{\theta _{0_{%
\mathcal{E}}}}}D_{\alpha }(Q,P_{0}).
\end{equation*}%
We prove that $\theta _{0}$ is the only parameter $\theta $ that satisfies $%
\inf_{Q\in \mathcal{M}_{\theta _{0_{E}}}}D_{\alpha }(Q,P_{0})=0$.\newline
Suppose that $\theta _{1}\neq \theta _{0}$ such that $\theta _{1}\in \arg
\inf_{\theta }\inf_{Q\in \mathcal{M}_{\theta _{0_{\mathcal{E}}}}}D_{\alpha
}(Q,P_{0})$.Then 
\begin{equation*}
\inf_{Q\in \mathcal{M}_{\theta _{1_{E}}}}D_{\alpha }(Q,P_{0})=\inf_{Q\in 
\mathcal{M}_{\theta _{0_{E}}}}D_{\alpha }(Q,P_{0})=0.\newline
\end{equation*}%
Since $\mathcal{M}_{\theta _{1_{\mathcal{E}}}}\subset \mathcal{M}_{\theta
_{1}}$\newline
\begin{equation*}
0=\inf_{Q\in \mathcal{M}_{\theta _{1_{E}}}}D_{\alpha }(Q,P_{0})\geq
\inf_{Q\in \mathcal{M}_{\theta _{1}}}D_{\alpha }(Q,P_{0})\geq 0.\newline
\end{equation*}%
Hence 
\begin{equation*}
\inf_{Q\in \mathcal{M}_{\theta _{1}}}D_{\alpha }(Q,P_{0})=0.
\end{equation*}%
\newline
Therefore $\theta _{0}$ is the only $\theta $ such that $P_{0}\in \mathcal{M}%
_{\theta _{0}}$ which proves that $\theta _{1}$ does not exist (otherwise $%
P_{0}=P_{\theta _{1}}$ due to (\ref{M1}).\newline
Second case: Suppose that $\ \ P_{0}=P_{\theta _{0}}\in \mathcal{M}$ and $%
P_{0}\notin \mathcal{M}_{\mathcal{E}}.$ Recall that%
\begin{equation*}
\theta _{0}=\arg \inf_{\theta }\inf_{Q\in \mathcal{M}_{\theta }}D_{\alpha
}(Q,P_{0}).
\end{equation*}%
\newline
We now show that 
\begin{equation*}
\theta _{0}=\arg \inf_{\theta }\inf_{Q\in \mathcal{M}_{\theta _{\mathcal{E}%
}}}D_{\alpha }(Q,P_{0}).
\end{equation*}%
.\newline
Project $P_{0}=P_{\theta _{0}}$ on $\mathcal{M}_{\mathcal{E}}$ and define 
\newline
\begin{equation*}
\theta _{1}\in \arg \inf_{\theta }\inf_{Q\in \mathcal{M}_{\theta _{\mathcal{E%
}}}}D_{\alpha }(Q,P_{0}).
\end{equation*}%
Assume that $\theta _{1}\neq \theta _{0}$ .\newline
We then have

\begin{equation*}
\inf_{Q\in \mathcal{M}_{\theta _{1_{\mathcal{E}}}}}D_{\alpha }(Q,P_{0})\leq
\inf_{Q\in \mathcal{M}_{\theta _{\mathcal{E}}}}D_{\alpha }(Q,P_{0})
\end{equation*}%
for all $\theta $ by definition of $\theta _{1}$ . So with $\theta =\theta
_{0}$ , we have 
\begin{equation}
\inf_{Q\in \mathcal{M}_{\theta _{1_{E}}}}D_{\alpha }(Q,P_{0})\leq \inf_{Q\in 
\mathcal{M}_{\theta _{0_{\mathcal{E}}}}}D_{\alpha }(Q,P_{0})  \label{contra}
\end{equation}%
Under (\ref{cond regul par rapp a E}) it holds $D_{\alpha }(\mathcal{M}%
_{\theta _{0_{E}}},P_{\theta _{0}})<D_{\alpha }(\mathcal{M}_{\theta _{%
\mathcal{E}}},P_{\theta _{0}})$, for all $\theta \neq \theta _{0}$.\newline
Hence (\ref{contra}) is impossible, so $\theta _{1}=\theta _{0}$ .We have
proved (\ref{Thm1}).

\subsection{Proof of Theorem \protect\ref{theoreme 3}}

Assume that $(Q_{n})_{n\geq 1}\subset \mathcal{M}_{\mathcal{E}}$ and assume
that there exists $q$ such that 
\begin{equation*}
\sup_{x}|q_{n}(x)-q(x)|\longrightarrow 0,
\end{equation*}%
with $q_{n}(x):=\left( dQ_{n}/d\lambda \right) (x).$ Define $%
Q(A):=\int_{A}q(x)\mathrm{d\lambda }(x)$for any set $A$ and we have $\left(
Q_{n}{\underset{st}{\longrightarrow }}Q\right) $ . 
We prove that $Q\in \mathcal{M}_{\mathcal{E}},$ stating

(A) $q$ is a density

(B) $\ \int_{K}g(x,\theta )q(x)\mathrm{d}x=0$ for some $\theta $

(C) $q$ satisfies (E1) and (E2).

We prove (A); 
This follows from Prohorov Theorem.

We prove (B); Let $\theta _{n}$ be defined by $\int g(x,\theta _{n})q_{n}(x)%
\mathrm{d}x=0;$ such a $\theta _{n}$ indeed exists since $Q_{n}\in \mathcal{M%
}$.\newline
Since $\Theta $ is a compact set in $\mathbb{R}^{d}$,we select ${n_{j}}%
\subset {n}$ such that the subsequence $\theta _{n_{j}}$ admits a limit $%
\underline{\theta }$ and $\int g(x,\theta _{n_{j}})q_{n_{j}}(x)\mathrm{d}x=0$
.\newline
We prove that $\left\vert \int_{K}g(x,\underline{\theta })q(x)\mathrm{d}%
x\right\vert =0$ .
\newline

Indeed 
\begin{align*}
\left\vert \int_{K}g(x,\underline{\theta })q(x)\mathrm{d}x\right\vert & \leq
\left\vert \int_{K}g(x,\underline{\theta })q_{n_{j}}(x)\mathrm{d}%
x\right\vert +\left\vert \int_{K}g(x,\underline{\theta })q(x)\mathrm{d}%
x-\int_{K}g(x,\underline{\theta })q_{n_{j}}(x))\mathrm{d}x\right\vert \\
& \leq B+A
\end{align*}%
\begin{equation*}
A=\left\vert \int_{K}g(x,\underline{\theta })\left( q(x)-q_{n_{j}}(x)\right) 
\mathrm{d}x\right\vert
\end{equation*}%
which tends to $0$ by ($\ref{G2}$).\newline
Next%
\begin{equation*}
B\leq \int_{K}\left\vert g(x,\underline{\theta })-g(x,\theta
_{n_{j}})\right\vert q_{n_{j}}(x)\mathrm{d}x+\left\vert \int_{K}g(x,\theta
_{n_{j}})q_{n_{j}}(x)\mathrm{d}x\right\vert \leq C+D
\end{equation*}%
and $D=0$ by definition of $\theta _{n_{j}}$. \newline
\newline
Hence

\begin{eqnarray*}
B &\leq &C=\int_{K}\left\vert g(x,\underline{\theta })-g(x,\theta
_{n_{j}})\right\vert q_{n_{j}}(x)\mathrm{d}x \\
&\leq &\sup_{x\in K}\left\vert g(x,\underline{\theta })-g(x,\theta
_{n_{j}})\right\vert \int_{K}q_{n_{j}}(x)\mathrm{d}x \\
&=&\sup_{x\in K}\left\vert g(x,\underline{\theta })-g(x,\theta
_{n_{j}})\right\vert \rightarrow 0.
\end{eqnarray*}%
We have proved that any converging sequence $\theta _{n_{j}}$ satisfies $%
\int_{K}g(x,\underline{\theta })q(x)\mathrm{d}x$ when $\underline{\theta }%
=\lim_{n_{j}\rightarrow \infty }\theta _{n_{j}}$ . \newline
Consider two converging subsequences ${n_{j}}$ and ${n_{j}^{\prime }}$ with $%
\theta _{n_{j}}\rightarrow \underline{\theta }$ and $\theta _{n_{j}}^{\prime
}\rightarrow \bar{\theta}$,we have 
\begin{equation*}
\int_{K}g(x,\underline{\theta })q(x)\mathrm{d}x=\int_{K}g(x,\bar{\theta})q(x)%
\mathrm{d}x.
\end{equation*}%
By ($\ref{M1}$) it follows that $\underline{\theta }=\bar{\theta}$ therefore
we have proved that \ there exists a unique $\theta \in \Theta $ such that 
\begin{equation*}
\int_{K}g(x,\theta )q(x)\mathrm{d}x=0
\end{equation*}%
which proves (B).

We prove (C); firstly there exists some $N>0$ such that $|q(x_{0})|\leq N$ .
Indeed \newline
\begin{equation*}
|q(x_{0})-q_{n}(x_{0})+q_{n}(x_{0})|\leq
|q_{n}(x_{0})|+|q_{n}(x_{0})-q(x_{0})|\leq N+|q_{n}(x_{0})-q(x_{0})|\leq
N+\varepsilon
\end{equation*}%
for all $\varepsilon >0$ and therefore $|q(x_{0})|\leq N$, $\ $since \newline
\begin{equation*}
|q_{n}(x_{0})-q(x_{0})|\rightarrow 0.
\end{equation*}%
This proves (E1).\ We prove that $q$ satisfies (\ref{Lipschitz unif pour q
alfa}).

Assume that there exists $\left( x,y\right) $ in $K$ such that $\left\vert
q^{\alpha }(x)-q^{\alpha }(y)\right\vert >M\left\vert x-y\right\vert .$ By
the triangle Inequality it then holds 
\begin{equation*}
\left\vert q^{\alpha }(x)-q^{\alpha }(y)\right\vert \leq \left\vert
q_{n}^{\alpha }(x)-q_{n}^{\alpha }(y)\right\vert +2\varepsilon _{n}
\end{equation*}%
where $\varepsilon _{n}:=\sup_{x\in K}\left\vert q^{\alpha
}(x)-q_{n}^{\alpha }(x)\right\vert \rightarrow 0$, whence $\left\vert
q_{n}^{\alpha }(x)-q_{n}^{\alpha }(y)\right\vert >M\left\vert x-y\right\vert
+\varepsilon _{n}^{\prime }$ with $\varepsilon _{n}^{\prime }$ $\rightarrow
0 $, a contradiction. Now either (E2) (i) pr (E2) (ii) clearly hold.

\subsection{Proof of Proposition \protect\ref{proposition 2}}

We prove that $A_{E}(a)$ is a closed subset in $\mathcal{M}_{\mathcal{E}}$
equipped with the strong topology . Recall that $Q\rightarrow D_{\alpha
}(Q,P)$ l.s.c is equivalent to $A_{\mathcal{E}}(a)$ is closed. \newline
Let $Q_{n}\in A_{\mathcal{E}}(a)\cap \mathcal{M}_{\mathcal{E}}$. Denote $%
\frac{dQ_{n}}{d\lambda }(x)=q_{n}(x)$ with $q_{n}\in E$,and assume that
there exists a function $q$ defined on $K$ such that 
\begin{equation*}
\sup_{x\in K}|q_{n}(x)-q(x)|\rightarrow 0.
\end{equation*}%
Define 
\begin{equation*}
\frac{dQ}{d\lambda }(x)=q(x).
\end{equation*}%
We prove that $q\in E$ and that $Q\in A_{\mathcal{E}}(a)$ with $Q(A):=\int
1_{A}(x)q(x)\mathrm{d}\lambda (x)$.

Since $\mathcal{M}_{\mathcal{E}}$ is closed, (see Theorem \ref{theoreme 3})
the measure $Q$ defined by $Q(A)=\int 1_{A}(x)q(x)\mathrm{d}\lambda (x)$ for
all $A\in \mathcal{B}(\mathbb{R})$ is in $\mathcal{M}_{\mathcal{E}}$.\newline
It remains to prove that $D_{\alpha }(Q,P)\leq a$.\newline

It holds $\sup_{x\in K}\left\vert q_{n}^{\beta }(x)-q^{\beta }(x)\right\vert
\rightarrow 0$ for all $\beta >0.$

\bigskip Consider now the mapping 
\begin{equation*}
x\rightarrow \varphi (q_{n}(x),p(x))-\varphi (q(x),p(x)).\newline
\end{equation*}%
Since 
\begin{equation*}
\varphi (q_{n}(x),p(x))-\varphi (q(x),p(x))=q_{n}^{\alpha +1}(x)-q^{\alpha
+1}(x)-\left( 1+\frac{1}{\alpha }\right) p(x)\left( q_{n}^{\alpha
}(x)-q^{\alpha }(x)\right) ;
\end{equation*}%
.\newline
it holds 
\begin{equation*}
\sup_{x\in K}|\varphi (q_{n}(x),p(x))-\varphi (q(x),p(x))|\rightarrow 0.
\end{equation*}%
\newline
Integrating we have

\begin{equation}
\int \varphi (q_{n}(x),p(x))\mathrm{d}x-\delta \leq \int \varphi (q(x),p(x))%
\mathrm{d}x=D_{\alpha }(Q,P)\leq \int \varphi (q_{n}(x),p(x))\mathrm{d}%
x+\delta  \label{Ineg1}
\end{equation}%
for any $\delta >0$, for $n$ large. 
Since $Q_{n}\in A_{\mathcal{E}}(a),\int \varphi (q_{n}(x),p(x))\mathrm{d}%
x\leq a$; the inequality (\ref{Ineg1}) becomes 
\begin{align*}
\int \varphi (q_{n}(x),p(x))\mathrm{d}x-\delta \leq \int \varphi (q(x),p(x))%
\mathrm{d}x\leq \int \varphi (q_{n}(x),p(x))\mathrm{d}x+\delta & \leq \\
& a+\delta
\end{align*}%
So $\int \varphi (q(x),p(x))\mathrm{d}x\leq a$; hence $Q\in A_{\mathcal{E}%
}(a)$ \ and thus $A_{\mathcal{E}}(a)$ is a closed set in $\mathcal{M}_{%
\mathcal{E}}$.

\subsection{Proof of Lemma \protect\ref{Lemma Cond 2 VdW}}

We thus prove Condition (2) in Lemma \ref{lemme 3}. By Proposition \ref%
{proposition 3} 
\begin{equation*}
Q_{\theta }^{\ast }:=\arg \inf_{Q\in \mathcal{M}_{\theta _{\mathcal{E}%
}}}R_{\alpha }(Q,P_{0})
\end{equation*}%
exists with uniqueness. Denote $q_{\theta }^{\ast }:=\frac{dQ^{\ast }(\theta
)}{d\lambda }$ . It holds 
\begin{equation*}
\inf_{\Vert q-q_{\theta }^{\ast }\Vert >\varepsilon ,q\in \mathcal{M}%
_{\theta _{\mathcal{E}}}}R_{\alpha }(Q,P_{0})>R_{\alpha }(Q^{\ast }(\theta
),P_{0}).
\end{equation*}%
Indeed by definition for all $Q$ , such that $\frac{dQ}{d\lambda }(x)=q(x)$ 
\begin{equation*}
R_{\alpha }(Q^{\ast }(\theta ),P_{0})\leq R_{\alpha }(Q,P_{0})
\end{equation*}%
and therefore 
\begin{equation*}
\text{\ }R_{\alpha }(Q^{\ast }(\theta ),P_{0})\leq \inf_{\Vert q-q^{\ast
}(\theta )\Vert >\varepsilon }R_{\alpha }(Q,P_{0}).
\end{equation*}%
Now let $Q^{\ast }(\theta )$ and denote $dQ^{\ast }(\theta )/d\lambda
(x)=q^{\ast }(\theta )(x)$ and $Q$ such that $dQ(\theta )/d\lambda
(x)=q(\theta )(x)$. We prove that the inequality is strict . From the above
display we get 
\begin{equation*}
R_{\alpha }(Q^{\ast }(\theta ),P_{0})+\frac{1}{\alpha }\int p_{0}^{\alpha
+1}(x)\mathrm{d}x\leq \inf_{\Vert q-q^{\ast }(\theta )\Vert >\varepsilon
}\left\{ R_{\alpha }(Q,P_{0})+\frac{1}{\alpha }\int p_{0}^{\alpha +1}(x)%
\mathrm{d}x\right\}
\end{equation*}

i.e. 
\begin{equation*}
D_{\alpha }(\mathcal{M}_{\theta _{\mathcal{E}}},P_{0})\leq \inf_{\Vert
q-q^{\ast }(\theta )\Vert >\varepsilon ,q\in \mathcal{M}_{\theta _{\mathcal{E%
}}}}D_{\alpha }(Q,P_{0}).
\end{equation*}%
Now if equality holds, there exists $q\in \mathcal{M}_{\theta _{\mathcal{E}%
}} $ , $q\neq q^{\ast }(\theta )$ such that 
\begin{equation}
D_{\alpha }(\mathcal{M}_{\theta _{\mathcal{E}}},P_{0})=D_{\alpha }(Q^{\ast
}(\theta ),P_{0})=D_{\alpha }(Q,P_{0}).  \label{fin preuve Lemma 6}
\end{equation}
It holds $Q\neq Q^{\ast }(\theta )$ since both $q^{\ast }(\theta )$ and $%
q\in E$ . But the projection of $P_{0}$ on $\mathcal{M}_{\theta _{\mathcal{E}%
}}$ is unique, so (\ref{fin preuve Lemma 6}) cannot hold. \newline

\subsection{Proof of Proposition \protect\ref{proposition 4}}

By Theorem \ \ref{theoreme 10} for all $\theta $%
\begin{equation*}
d(q_{n}(\theta ),q_{\theta }^{\ast })\rightarrow 0\text{ in probability.}
\end{equation*}%
\bigskip We want to prove that uniform convergence upon $\theta $ holds.
Define $\theta _{n}$ $\ $by%
\begin{equation}
sup_{\theta \in \Theta }d(q_{n}(\theta ),q_{\theta }^{\ast })=d(q_{n}(\theta
_{n}),q_{\theta _{n}}^{\ast }).\newline
\label{teta_n}
\end{equation}%
\newline
Let $\{n_{j}\}\subset \{n\}$ and suppose $\underline{\theta }$ such that $%
\theta _{n_{j}}\rightarrow \underline{\theta }.$\newline
We show that $d(q_{n_{j}}(\theta _{n_{j}}),q_{\theta _{n_{j}}}^{\ast })>c>0$
cannot hold. \newline
By definition (\ref{teta_n}) 
\begin{eqnarray*}
sup_{\theta \in \Theta }d(q_{n_{j}}(\theta ),q_{\theta }^{\ast })
&=&d(q_{n_{j}}(\theta _{n_{j}}),q_{\theta _{n_{j}}}^{\ast }) \\
&\leq &d(q_{n_{j}}(\theta _{n_{j}}),q_{n_{j}}(\underline{\theta }%
))+d(q_{n_{j}}(\underline{\theta }),q_{\underline{\theta }}^{\ast
})+d(q_{\theta _{n_{j}}}^{\ast },q_{\underline{\theta }}^{\ast }) \\
&=&:I_{1}+I_{2}+I_{3}.\newline
\end{eqnarray*}%
Now $I_{1}=d(q_{n_{j}}(\theta _{n_{j}}),q_{n_{j}}(\underline{\theta }))$ and 
$d(\theta _{n_{j}},\underline{\theta })\rightarrow 0.$ Hence under (\ref{K}%
), $I_{1}{\overset{P}{\longrightarrow }}0.$ Now $I_{2}=d(q_{n_{j}}(%
\underline{\theta }),q_{\underline{\theta }}^{\ast })$ ; both $q_{n_{j}}(%
\underline{\theta })$ and $q_{\underline{\theta }}^{\ast }$ belong to $%
\mathcal{M}_{\underline{\theta }_{\mathcal{E}}}$ ; By Theorem \ref{theoreme
10} in $\mathcal{M}_{\underline{\theta }_{\mathcal{E}}}$ , $d(q_{n_{j}}(%
\underline{\theta }),q_{\underline{\theta }}^{\ast }){\overset{P}{%
\longrightarrow }}0$ so $I_{2}{\overset{P}{\longrightarrow }}0,$\newline
As for $I_{3}=d(q_{\underline{\theta }}^{\ast },q_{\theta _{n_{j}}}^{\ast }){%
\overset{P}{\longrightarrow }}0$ , as for $I_{1}.$ We have proved that%
\begin{equation}
\lim_{j\rightarrow \infty }\sup_{\theta \in \Theta }d\left( q_{n_{j}}(\theta
),q_{\theta }^{\ast }\right) =0\text{ in probability.}
\label{fin dem prop 4}
\end{equation}%
Assume now that (\ref{teta_n}) does not hold. In such a case there exists a
subsequence $\left\{ m_{k}\right\} \subset \left\{ n\right\} $ and $\eta >0$
such that 
\begin{equation*}
\sup_{\theta }d(q_{m_{k}}(\theta ),q_{\theta }^{\ast })>\eta .
\end{equation*}%
Let $\theta _{m_{k}}:=\arg \sup_{\theta }d(q_{m_{kj}}(\theta ),q_{\theta
}^{\ast })$, whence 
\begin{equation*}
d(q_{m_{k}}(\theta _{n_{j}}),q_{\theta _{m_{k}}}^{\ast })>\eta
\end{equation*}%
for all $k$. Extract from $\left\{ m_{k}\right\} $ a further subsequence $%
\left\{ n_{j}\right\} $ along which $\theta _{n_{j}}$ converges to some $%
\underline{\theta }.$ Then (\ref{fin dem prop 4}) proves our claim, by
contradiction.

\subsection{Proof of Lemma \protect\ref{Lemma 8}}

Define 
\begin{equation*}
M_{n}(\theta )=R_{\alpha }(q_{n}(\theta ),P_{n}),\text{ and }M(\theta
)=R_{\alpha }(q_{\theta }^{\ast },P_{0})
\end{equation*}%
with 
\begin{equation*}
R_{\alpha }(q_{n}(\theta ),P_{n})=\int q_{n}^{\alpha +1}(\theta )(x)\mathrm{d%
}x-\left( 1+\frac{1}{\alpha }\right) \int q_{n}^{\alpha }(\theta )(x)\mathrm{%
d}P_{n}(x)
\end{equation*}%
and%
\begin{equation*}
R_{\alpha }(q_{\theta }^{\ast },P_{0})=\int q_{\theta }^{\ast \alpha +1}(x)%
\mathrm{d}x-\left( 1+\frac{1}{\alpha }\right) \int q_{\theta }^{\ast \alpha
}(x)\mathrm{d}P_{0}(x)
\end{equation*}

Hence 
\begin{align*}
\sup_{\theta \in \Theta }\left\vert M_{n}(\theta )-M(\theta )\right\vert &
\leq \sup_{\theta \in \Theta }\int \left\vert q_{n}^{\alpha +1}(\theta
)(x)-q_{\theta }^{\ast \alpha +1}(x)\right\vert \mathrm{d}x+ \\
& \left( 1+\frac{1}{\alpha }\right) \sup_{\theta \in \Theta }\int \left\vert
q_{n}^{\alpha }(\theta )(x)-q_{\theta }^{\ast \alpha }(x)\right\vert \mathrm{%
d}P_{n}(x) \\
& +\left( 1+\frac{1}{\alpha }\right) \sup_{\theta \in \Theta }\left\vert
\int q_{\theta }^{\ast \alpha }(x)\mathrm{d}(P_{n}-P_{0})\right\vert \\
& \leq R_{1}+R_{2}+R_{3}.
\end{align*}%
Now 
\begin{eqnarray*}
R_{1} &=&\sup_{\theta \in \Theta }\int \left\vert q_{n}^{\alpha +1}(\theta
)(x)-q_{\theta }^{\ast \alpha +1}(x)\right\vert \mathrm{d}x \\
&\leq &\sup_{\theta \in \Theta }\sup_{x\in K}\left\vert q_{n}(\theta
)(x)-q_{\theta }^{\ast }(x)\right\vert \times Cste
\end{eqnarray*}%
which tends to $0$ in probability by Proposition \ref{proposition 4}.

Also 
\begin{eqnarray*}
R_{2} &=&\left( 1+\frac{1}{\alpha }\right) \sup_{\theta \in \Theta }\int
\left\vert q_{n}^{\alpha }(\theta )(x)-q_{\theta }^{\ast \alpha
}(x)\right\vert \mathrm{d}P_{n}(x) \\
&\leq &\sup_{\theta \in \Theta }\left( 1+\frac{1}{\alpha }\right) \frac{1}{n}%
\sum \left\vert q_{n}^{\alpha }(\theta )(X_{i})-q_{\theta }^{\ast \alpha
}(X_{i})\right\vert
\end{eqnarray*}

\begin{equation*}
\leq Cste\times \sup_{\theta \in \Theta }\left\vert q_{n}(\theta
)(x)-q_{\theta }^{\ast }(x)\right\vert
\end{equation*}%
which tends to $0$ in probability, making use of Proposition \ref%
{proposition 4} .

Turn to $R_{3}.$ The class of functions $q_{\theta }^{\ast \alpha }$ indexed
by $\theta $ satisfies the three following properties: (i) It is indexed by $%
\theta $ in $\Theta $, a compact subset of $\mathbb{R}^{d}$.(ii) Secondly it
is continuous in $\theta $ for all $x$ in $K$. (iii) Thirdly the function $F$
defined on $K$ by $F(x):=\sup_{\theta \in \Theta }\left\vert q_{\theta
}^{\ast \alpha }(x)\right\vert $ is such that 
\begin{equation*}
\int F(x)dP_{0}(x)<\infty .
\end{equation*}%
Whenever these three facts hold, then 
\begin{equation*}
R_{3}=\left( 1+\frac{1}{\alpha }\right) \sup_{\theta \in \Theta }\left\vert
\int q_{\theta }^{\ast \alpha }(x)\mathrm{d}(P_{n}-P_{0})\right\vert
\end{equation*}%
tends to $0$ in Probability since $\{q_{\theta }^{\ast \alpha }\}_{\theta }$
is a Glivenko-Cantelli class of functions, making use of \cite{Wellner 2005}%
,Chapter 1.6.


\begin{acknowledgement}
The authors are very grateful to the Editor and two anonymous reviewers for
their acute and very constructive remarks and suggestions which helped to
improve on the initial draft.
\end{acknowledgement}

\end{document}